\documentclass[leqno,10pt]{amsart}

\usepackage{amsmath}

\usepackage{amssymb}
\usepackage{graphicx}

\usepackage{color}
\usepackage[dvipsnames]{xcolor}

\newtheorem{thm}{Theorem}

\newtheorem{thmx}{Theorem} 

\newtheoremstyle{lemmma}{4mm}{1mm}{\itshape}{ }{\bfseries}{.}{ }{}
\newtheorem{lemmaA}{Lemma A\ignorespaces}
\newtheorem{lemmaB}{Lemma B\ignorespaces}

\newtheorem{prop}[thm]{Proposition}

\newtheorem{lemma}[thm]{Lemma}

\theoremstyle{definition}
\newtheorem{defn}{Definition}
\newtheorem{ex}{Example}

\theoremstyle{remark}
\newtheorem{remark}{Remark}
\newtheorem{notation}{Notation}

\numberwithin{equation}{section}

\def\R{\mathbb{R}}

\def\Z{\mathbb{Z}}

\def\O{\mathrm{O}^{\uparrow}_{+}}
\def\OO{\widehat{\mathrm{O}}^{\uparrow}_{+}}
\def\UO{\widetilde{\mathrm{O}}^{\uparrow}_{+}}
\def\GL{\mathrm{GL}}

\def\SS{\mathbb{S}}

\def\V{\mathbb{V}}

\def\NG2{\mathrm{G}^-_2(\R^{2,n+1})}
\def\MK2{\R^{2,n+1}}
\def\E{\mathcal{E}^{1,n}}
\def\EC{\mathcal{E}^{1,n}_{\mathrm{I}}}
\def\ECC{\mathcal{E}^{1,n}_{\mathrm{II}}}

\begin{document}

%
\title[]{On the restricted conformal group \\of the $(1+n)$-Einstein static universe}

\author{Olimjon Eshkobilov}
\address{(O. Eshkobilov) Dipartimento di Matematica ``Giuseppe Peano'', Universit\`a di Torino,
Via Carlo Alberto 10, I-10123 Torino, Italy}
\email{olimjon.eshkobilov@edu.unito.it, olimjon.eshkobilov@polito.it
}

\author{Emilio Musso}
\address{(E. Musso) Dipartimento di Scienze Matematiche, Politecnico di Torino,
Corso Duca degli Abruz\-zi 24, I-10129 Torino, Italy}
\email{emilio.musso@polito.it}

\author{Lorenzo Nicolodi}
\address{(L. Nicolodi) Di\-par\-ti\-men\-to di Scienze Ma\-te\-ma\-ti\-che, Fisiche e Informatiche,
Uni\-ver\-si\-t\`a di Parma, Parco Area delle Scienze 53/A,
I-43124 Parma, Italy}
\email{lorenzo.nicolodi@unipr.it}

\thanks{Authors partially supported by
PRIN 2015-2018 ``Variet\`a reali e complesse: geometria, to\-po\-lo\-gia e analisi ar\-mo\-ni\-ca'';
by the GNSAGA of INDAM; and by the FFABR Grant 2017 of MIUR.
The present research was also partially supported by MIUR grant
``Dipartimenti di Eccellenza'' 2018–2022, CUP: E11G18000350001, DISMA, Politecnico
di Torino.}





\begin{abstract}

%
%

Explicit models for the restricted conformal group of the Einstein static universe of dimension
greater than two and for its universal covering group are constructed.
Based on these models, as an application
we determine all oriented
and time-oriented conformal Lorentz manifolds whose restricted conformal group has maximal dimension.
They amount to the Einstein static universe itself and two countably infinite series of its compact
quotients.


%
\end{abstract}

\maketitle

\section{Introduction}\label{s:0}

Let $\EC$ and $\ECC$ denote the conformal compactifications of Minkowski $(n+1)$-space ($n\geq 2$),
realized respectively as the sets
%
%
of oriented and unoriented null lines through the origin in
pseudo-Euclidean space
$\mathbb R^{2,n+1}$.
Endowed with the
Lorentz structures inherited from $\mathbb R^{2,n+1}$,
they are known in the literature with the name of (compact) $(1+n)$-{\em Einstein universes}.
Topologically, $\EC \cong  \SS^1 \times \SS^n$ and $\ECC \cong  \SS^1 \times \SS^n/\{\pm 1\}$,
where $-1$ acts by the antipodal map on both factors. The space $\EC$ is orientable. Instead,
if $n$ is even,
$\ECC$ is not orientable, and $\EC$ is an orientable double covering of $\ECC$.
%
%
The $(1+n)$-{\em Einstein static universe} is
the product space $\E \cong \R\times \SS^n$ with the Lorentz product metric
$-dt^2 + g_{\SS^n}$, where $g_{\SS^n}$ denotes the standard metric of $\SS^n$
(cf. \cite{BCDG, Ei, HE}). The space $\E$ is the universal covering of both $\EC$ and
$\ECC$.



\vskip0.2cm

The Einstein universes $\EC$, $\ECC$ ($n$ odd), and $\E$ are examples of
oriented, time-oriented, conformal Lorentz manifolds of dimension $n+1$.
For an oriented, time-oriented, conformal Lorentz
manifold $(\mathbb M,[g])$ of dimension $n + 1 \geq 3$, we let $\mathcal{C}^{\uparrow}_+(\mathbb{M})$
denote the \textit{restricted conformal group} of $\mathbb M$, i.e., the group of conformal
transformations preserving orientation and time-orientation.
%
It is well-known that
%
%
$\mathcal{C}^{\uparrow}_+(\EC) \cong\O(2,n+1)$ and $\mathcal{C}^{\uparrow}_+(\ECC)\cong\O(2,n+1)/\{\pm I\}$,
where $\O(2,n+1)$ is the identity component of $\mathrm O(2,n+1)$, the pseudo-orthogonal group
of a scalar product of signature $(2,n+1)$ (cf. \cite{CaKe, Cecil} and Section \ref{s:2:s:3}).
As for $\E$, it is known that
$\mathcal{C}^{\uparrow}_+(\E) \cong \OO(2,n+1)$, where $\OO(2,n+1)$
is a
central extension of $\O(2,n+1)$ which can be obtained as a
quotient of index 2 of the universal covering group $\UO(2,n+1)$ of $\O(2,n+1)$.
In the following, $\OO(2,n+1)$ will be referred to as the \textit{canonical covering} of $\O(2,n+1)$.
It is important to observe that:
(1) both the Lie groups $\OO(2,n+1)$ and $\UO(2,n+1)$
do not have faithful finite dimensional
representations, so that there are no nice models for them as Lie groups of matrices;\footnote{In this respect,
observe that $\UO(2,3)\cong \widetilde{\mathrm{Sp}}(4,\mathbb R)$, the universal covering group of
${\mathrm{Sp}}(4,\mathbb R)$.
}
(2) the restricted conformal groups of
$\EC$, $\ECC$ ($n$ odd), and $\E$
have the largest possible dimension for a Lorentz manifold
of dimension $n+1$, namely $N$ $=$ $(n+3)(n+2)/2$ (cf. Section \ref{s:1} for more theoretical details on the restricted conformal group
of a Lorentz manifold).

\vskip0.2cm

The purposes of this paper are twofold. The first is to provide explicit
models for
the Lie groups $\OO(2,n+1)$ and $\UO(2,n+1)$, that is, to describe the
underlying group manifolds
and the respective group multiplications.
The second purpose is to use the models for $\OO(2,n+1)$ and $\UO(2,n+1)$
to address the question of the
characterization of
oriented, time-oriented, conformal Lorentz manifolds $(\mathbb M^{n+1}, [g])$
($n \geq 2$)
whose {restricted conformal group} $\mathcal{C}^{\uparrow}_+(\mathbb{M})$
has maximal dimension.
\vskip0.2cm
The main results of the paper are presented in three theorems, which we now illustrate individually.

\vskip0.1cm
Theorem A provides explicit models for the canonical covering $\OO(2,n+1)$
and for the universal covering $\UO(2,n+1)$.
The underlying group manifolds are described explicitly
as hypersurfaces in $\OO(2,n+1)\times \mathbb R$ and in $\mathrm{Spin}(2,n+1)\times \R$,
and the respective group multiplications are given by single global formulas.
%
Our approach was inspired by
%
the basic construction of
a manifold underlying the universal covering
of a Lie group $G$ with
$\pi_1(G) = \mathbb Z$ given in \cite{Ra},
and
by the construction of nontrivial central extensions of the real symplectic group $\mathrm{Sp}(2n,\R)$,
such as the circle extension $\mathrm{Mp}^{\mathrm c}(2n,\R)$ and
the universal covering group
${\widetilde{\mathrm{Sp}}(2n,\R)}$ (cf. \cite{Ra, RoRa}).
The idea of this method, in turn, has its origin in the classical work of Bargmann \cite{Ba, Ba1} on the
irreducible unitary representations of the Lorentz group.

\vskip0.1cm
Theorem B proves that the Lie group $\OO(2,n+1)$ constructed in Theorem A is indeed isomorphic
to the restricted conformal
group of the Einstein static universe $\E$.
It also describes the restricted conformal groups of
two countably infinite series
of compact quotients of $\E$, namely
$\mathcal{E}^{1,n}_{\mathrm{I},k}$ ($k \geq 1$) and $\mathcal{E}^{1,n}_{\mathrm{II},k}$ ($k \geq 0$; $n$ odd),
called the {\em integral compact forms of the first and second kind with index $k$}, respectively.
In particular, $\mathcal{E}^{1,n}_{\mathrm{I},1}=\EC$ and $\mathcal{E}^{1,n}_{\mathrm{II},0}= \ECC$,
which are referred to as the {\em standard compact forms}.
Interestingly enough, the restricted conformal group of any integral compact form attains
the maximum dimension $N$.
It is important to observe that
two integral compact forms with different indices in the same series,
as well as
two integral compact forms in different series, cannot be conformally
equivalent. This is proved in Proposition \ref{prop:inequiv}.
%
%

\vskip0.1cm
Theorem C proves that
if a connected, oriented, time-oriented, conformal Lorentz manifold
of dimension $n+1$ ($n\geq 2$) has a restricted conformal group of maximal dimension $N$,
then it is conformally equivalent to either $\E$
or to an integral compact form
$\mathcal{E}^{1,n}_{\mathrm{I},k}$ or $\mathcal{E}^{1,n}_{\mathrm{II},k}$.
%
The above characterization provides
conformally nonequivalent geometric models
for the
Lorentz manifolds with an essential\footnote{The conformal group of a Lorentz manifold $(\mathbb M , [g])$
is said to be essential if it is strictly larger than the isometry group of any metric in the conformal class of $g$
(cf. \cite{AL1,AL3}).}
restricted conformal group
of maximal dimension. Observe in particular that these models are
all conformally flat.
An interesting problem is to investigate
the possibility of providing
for the examples studied by
Alekseevsky \cite{AL2} suitable geometric models which are
locally but not globally conformally equivalent to each other.

We conclude by recalling that
in the Riemannian case, in contrast to the result of Theorem C,
a conformal Riemmannian manifold $(\mathbb M,[g])$ of dimension $n \geq 3$ with
a conformal group of maximal dimension must be conformally
diffeomorphic to $\SS^n$ with its natural conformal structure (cf. \cite{K1, LeFe, Ob}).

\vskip0.2cm

The paper is organized as follows. Section \ref{s:1} collects and reformulates some known
facts about conformal Lorentz geometry
for the purpose of the discussion.
%
More specifically, it describes the
Einstein static universe and its compact quotients,
and among the compact quotients distinguishes two
countably infinite families, the {integral compact forms} of the first and second kind.
The construction of the Cartan conformal bundle
and of the normal conformal
Cartan connection for an oriented, time-oriented, conformal Lorentz manifold $\mathbb M$ of dimension $n+1\geq 3$
is then briefly recalled.
%
%
A classical result of Cartan and Kobayashi on the conformal group of a
Riemannian manifold (cf. \cite{Ca3,K1})
is then extended to the case of the restricted conformal group of a Lorentz manifold.
This technical result,
which was indeed largely predictable,
%
will
play an important role in the proofs of Theorems A and B.
Finally, the special examples
$\EC$ and $\ECC$ are discussed.

\vskip0.1cm

Section \ref{s:3} constructs the canonical covering $\OO(2,n+1)$ of $\O(2,n+1)$ using the transitive action of $\O(2,n+1)$ on the irreducible
bounded symmetric domain of type IV, regarded as a homogeneous space of $(n+1) \times 2$ real matrices
(cf. \cite{BO, Ca1, Hel, Hua, HuaBook, Sat}).
The group manifold of the canonical covering is realized explicitly as an embedded submanifold
of the product $\O(2,n+1)\times \R$ and the group multiplication defining the Lie group structure is given
by a single global formula. The center of $\OO(2,n+1)$ is also computed. This is the content of
Theorem \ref{ThmA}.
By a similar construction, an explicit realization of the universal covering group of $\O(2,n+1)$ is obtained.
We then prove that $\OO(2,n+1)$ is isomorphic to the restricted conformal group of the
Einstein static universe and describe the restricted conformal groups of the integral compact
forms. This is the content of Theorem \ref{ThmB}.

\vskip0.1cm
Section \ref{s:4} proves that if the restricted conformal group of an oriented, time-oriented Lorentz
manifold $\mathbb{M}$ has maximal dimension, then $\mathbb{M}$ is conformally equivalent to either
the simply connected Einstein static universe, or to one of its integral compact forms. This is
the content of Theorem \ref{ThmC}.

\vskip0.1cm


\section{The Einstein static universe and the integral compact forms}\label{s:1}

In this section we
introduce two infinite series of compact quotients of $\E$, the so-called {\em integral compact forms
of the first and second kind}. These include as special cases
the compact Einstein universes $\EC$ and $\ECC$ ($n$ odd),
also called {\em standard compact forms of the
first and second kind}.
We then present a direct construction
of the Cartan conformal bundle and of the normal conformal Cartan connection
for a conformal Lorentz manifold
without resorting to the abstract theory of prolongation of $G$-structures.
We reformulate, for conformal Lorentz manifolds,
a classical result on the conformal group of a Riemannian manifold.
As an example, we describe the restricted conformal groups of the standard
forms $\EC$ and $\ECC$ ($n$ odd).

\subsection{Preliminaries}\label{s:1:s:0}
For given integers $p,q$, $1\le p\le 2$, $p<q$, let $\R^{p,q}$ denote $\R^m$, $m=p+q$, with the
nondegenerate scalar product
%
%
\begin{equation}\label{s:1:s:0:f1}
 \langle x,y\rangle = -x^0y^0 + (-1)^{p-1}x^1y^1+\sum_{j=2}^{m-1}x^jy^j = {}^t\!x G y
  \end{equation}
of signature $(p,q)$, where $x^0,\dots,x^{m-1}$ denote the coordinates with respect to the standard basis
$e=(e_0, \dots, e_{m-1})$ of $\R^m$.
Let $\O(p,q)$ denote the identity
component of the pseudo-orthogonal group of \eqref{s:1:s:0:f1}.
For the purpose of conformal geometry, instead of the canonical coordinates, it is convenient
to use the coordinates $u = {}^t\!(u^0,\dots,u^{m-1})$ defined by
$u = \mathcal D_px$, where $\mathcal{D}_p\in \mathrm{GL}(m,\R)$ is given by
\[
 {\sqrt{2}}\mathcal{D}_p := E^0_0 + (-1)^{p-1}E^0_{m-1} + E^{m-1}_0- (-1)^{p-1}E^{m-1}_{m-1}
 +{\sqrt{2}}\sum_{j=1}^{m-2}E_j^j.
  \]
Here $E^h_k$,
$0\le h,k \le m-1$, denotes the elementary $m\times m$ matrix with 1
in the $(h,k)$ place and zero elsewhere.
%
%
%
In the coordinates $u = {}^t\!(u^0,\dots,u^{m-1})$ the scalar product
\eqref{s:1:s:0:f1} takes the form
\begin{equation}\label{moebius-sp}
  -u^0v^{m-1} - u^{m-1}v^0 + (-1)^{p-1}u^1v^1 +\sum_{j=2}^{m-2}u^jv^j.
  \end{equation}
The pseudo-orthogonal group $\mathrm{M}^{\uparrow}_+(p,q)$ of \eqref{moebius-sp} is
the image of the faithful representation of $\O(p,q)$ given by
\begin{equation}\label{IS1}
\chi :  \O(p,q)\ni \mathbf{X} \longmapsto  \mathcal{D}_p  \mathbf{X}  \mathcal{D}_{p}^{-1}\in \mathrm{GL}(m,\R),
  \end{equation}
that is, $\mathrm{M}^{\uparrow}_+(p,q) = \mathcal{D}_p  \O(p,q) \mathcal{D}_{p}^{-1}$.
The Lie algebra of $\mathrm{M}^{\uparrow}_+(p,q)$ will be denoted by $\mathfrak{m}(p,q)$.

\vskip0.1cm
Let $\mathrm{H}^{\uparrow}_+(p,q)\subset \mathrm{M}^{\uparrow}_+(p,q)$ be the parabolic subgroup
\begin{equation}\label{prb}
 \mathrm{H}^{\uparrow}_+(p,q):
 %
 =\{\mathbf{X} \in \mathrm{M}^{\uparrow}_+(p,q) \mid \mathbf{X}e_0 = r e_0, \, r\in \R, \, r>0\}.
  \end{equation}
The elements of $\mathrm{H}^{\uparrow}_+(p,q)$ can be written as
\[
\mathrm{\mathbf{X}}(r,B,y)=\left(
                     \begin{array}{ccc}
                       r& {}^*\!y  \mathrm{B}  & {^*\!y  y}/{2r}\\
                       0 & \mathrm{B} & {y}/{r} \\
                       0 & 0 & {r}^{-1} \\
                     \end{array}
                   \right),
                    \]
where $r>0$, $y={}^t\!(y^1,\dots,y^{m-2})\in \R^{m-2}$, $\mathrm{B}\in \O(p-1,q-1)$ and ${}^*\!y=\left((-1)^{p-1}y^1,y^2,\dots,y^{m-2}\right)\in (\R^{m-2})^*$.\footnote{If $p=1$, then $\O(p-1,q-1)$ is the special orthogonal group $\mathrm{SO}(q-1)$.} Let $\mathfrak{h}(p,q)$ be the Lie algebra of $\mathrm{H}^{\uparrow}_+(p,q)$.

\subsection{The Einstein static universe}\label{s:1:s:1}
The $(1+n)$-dimensional Einstein static universe $\E$ is the hypersurface of $\R^{1,n+1}\cong \R\times \R^{n+1}$ defined by
\[
 \E :=\left\{(\tau,x)\in \R^{1,n+1} \mid  {}^t\!x x = 1 \right\}\cong \R\times \SS^n,
  \]
equipped with the Lorentz metric
\[
  \widehat{\ell}_{\mathcal{E}} = -d\tau^2+\imath^\ast\Big(\sum_{j=1}^{n+1}(dx^j)^2\Big),
   \]
where $\imath : \SS^n \hookrightarrow \R^{n+1}$ denotes the inclusion map.
On $\E$ we consider the time-orientation given by requiring that the unit timelike vector field $\partial_{\tau}$
is future-directed, and the orientation induced by the volume form
\[
  {\Omega_{\widehat{\mathcal{E}}}}{|_{(\tau,x)}}
    =(d\tau\wedge \imath^\ast_{x}(dx^1\wedge\cdots\wedge dx^{n+1}))|_{T_{(\tau,x)}(\E)}.
    \]
The motivation for this terminology is that $\E$ is a static solution of Einstein's equation with a positive
cosmological constant \cite{CGP-BAMS, HE}. This solution was proposed by Einstein himself as
a model of a closed universe filled
with a perfect fluid of constant pressure and energy density \cite{Ei}. The physical relevance of the Einstein
universe is due to the fact that every
Friedmann--Robertson--Walker spacetime can be conformally embedded in $\E$ \cite{HE}.

\subsection{Compact quotients: The integral compact forms}\label{s:1:s:2}

Let $\mathbf{\tau}_*>0$ be a positive real number.
The subgroup $\mathcal{T}_{\tau^*}$ of the isometry group $\R\times \mathrm{SO}(n+1)$
generated by the translation
$\mathrm{T}_{\tau_*}:(\tau,\mathrm{y})\mapsto (\tau+\tau_*,\mathrm{y})$ acts properly discontinuously on $\E$.
This action preserves the Lorentz metric $\widehat{\ell}_{\mathcal{E}}$, the volume element
$\Omega_{\widehat{\mathcal{E}}}$, and the time-orientation.
The quotient manifold $\E/\mathcal{T}_{\tau^*}$ possesses a unique
Lorentz metric $\ell_{\mathcal{E},\tau_*}$, and a unique orientation and time-orientation,
such that the covering map
$\pi_{\tau_*}: \E\to \E/\mathcal{T}_{\tau^*}$ is a local isometry preserving orientation and time-orientation.

\begin{defn}\label{d:icf1stk}
If $\tau_*=2k\pi$, $k$ a positive integer, the quotient manifold $\E/\mathcal{T}_{\tau^*}$, denoted by
$\mathcal{E}^{1,n}_{\mathrm{I},k}$, is referred to as the {\it integral compact form of the first kind with index $k$}.
In this case, the covering $\pi_{\tau_*}$
will be denoted by $\pi_{\mathrm{I},k}$.
When $k=1$, the integral compact form $\mathcal{E}^{1,n}_{\mathrm{I},1}$ coincides with $\mathcal{E}^{1,n}_{\mathrm{I}}$,
and is referred to as the {\it standard compact form of the first kind}.
%
The integral forms $\mathcal{E}^{1,n}_{\mathrm{I},k}$ are all diffeomorphic to $\SS^1 \times \SS^n
\subset \mathbb C \times \mathbb R^{n+1}$.
\end{defn}

\begin{remark}

The standard compact form $\mathcal{E}^{1,n}_{\mathrm{I}} = (\mathbb R/2\pi\mathbb Z) \times \mathbb S^n
\cong \mathbb S^1 \times \mathbb S^n$.
Here $\mathbb S^1$
is the unit circle viewed as
a multiplicative subgroup of $\mathbb C$,
and the isomorphism $\mathbb R/2\pi\mathbb Z \cong \mathbb S^1$ is induced by $t \mapsto e^{it}$.
The conformal structure on $\mathcal{E}^{1,n}_{\mathrm{I}}$ is given by the Lorentz metric
$
-d\theta^2 + g_{\mathbb S^n}$, where
$\theta: \mathbb R/2\pi\mathbb Z \to \mathbb S^1$ is the argument function.
For a positive integer $k$,
the map $p_k :
\SS^1 \times \SS^n\ni (\mathrm x, \mathrm y) \longmapsto ({\mathrm x}^k, \mathrm y)\in \SS^1 \times \SS^n$
defines a $k:1$ covering.
%
%
Therefore, the integral form $\mathcal{E}^{1,n}_{\mathrm{I},k}$ can be thought of as
$\SS^1 \times \SS^n$ with the conformal structure induced by the Lorentz metric $\ell_{\mathcal{E},k} =
 {p_k}^\ast(-d\theta^2 + g_{\mathbb S^n})= -k^2 d\theta^2 + g_{\mathbb S^n}$.
Notice that
$p_k : \mathcal{E}^{1,n}_{\mathrm{I},k} \to \mathcal{E}^{1,n}_{\mathrm{I}}$
is a $k:1$ Lorentzian covering map.

\end{remark}

If $n$ is odd, the map $\mathrm{T}'_{\tau_*}:(\tau,\mathrm{y})\mapsto (\tau+\tau_*,-\mathrm{y})$ is an isometry
that preserves the orientation and the time-orientation. The subgroup $\mathcal{T}'_{\tau^*}$ generated
by $\mathrm{T}'_{\tau_*}$ acts properly discontinuously on $\E$. Consequently, the quotient manifold $\E/\mathcal{T}'_{\tau^*}$
inherits a unique Lorentz metric $\ell'_{\mathcal{E},\tau_*}$ and a unique orientation and time-orientation such
that the covering map
$\pi'_{\tau_*}:\E\to
\E/\mathcal{T}'_{\tau^*}$ is a local isometry preserving
orientation and time-orientation.

\begin{defn}\label{d:icf2ndk}
If $n$ is odd and $\tau_*= (2k+1)\pi$, $k$ a non-negative integer, the quotient manifold $\E/\mathcal{T}'_{\tau^*}$,
denoted by $\mathcal{E}^{1,n}_{\mathrm{II},k}$, is called the {\it integral compact form of the second kind with index $k$}.
In this case, the covering $\pi'_{\tau_*}$
will be denoted by $\pi_{\mathrm{II},k}$.
When $k=0$, the integral compact form $\mathcal{E}^{1,n}_{\mathrm{II},0}$ coincides with $\mathcal{E}^{1,n}_{\mathrm{II}}$,
and is referred to as the {\it standard compact form of the second kind}.
Notice that also the integral forms $\mathcal{E}^{1,n}_{\mathrm{II},k}$ are diffeomorphic to $\SS^1 \times \SS^n$.
If $n$ is even, this is not true anymore.
\end{defn}


\subsection{The Cartan conformal bundle}\label{s:2:s:1}

Let us start by recalling some definitions and fixing some notation.

\begin{defn}

Two Lorentz metrics $g$ and $g'$ on a manifold $\mathbb M$ are said to be {\em conformal} to each other if
$g' = r^2 g$, for some smooth function $r : \mathbb M \to \mathbb R$. The conformal class of $g$ is
$[g] = \{g' \mid g' \,\,\text{is conformal to}\,\, g\}$.
A {\em conformal Lorentz structure} on $\mathbb M$ amounts to the assignment of a conformal class $[g]$ of Lorentz metrics
on $\mathbb M$.

A {\em conformal Lorentz manifold} $(\mathbb M, [g])$
is an oriented, connected smooth manifold $\mathbb{M}$ of dimension $n+1\ge 3$ with a conformal Lorentz structure $[g]$.
%
%
We assume that the conformal structure is time-orientable, i.e., that the bundle $\dot{\mathrm{N}}(\mathbb{M})$ of
nonzero timelike tangent vectors
is disconnected. A  time-orientation is defined by the choice of a
connected component
$\dot{\mathrm{N}}^\uparrow(\mathbb{M})$
of $\dot{\mathrm{N}}(\mathbb{M})$.

A {\em conformal transformation} of $(\mathbb M, [g])$ is a diffeomorphism $F : \mathbb M \to \mathbb M$
that preserves the conformal class $[g]$, that is, $F^\ast ([g]) = [g]$. A {\em restricted conformal transformation}
is a conformal transformation
which, in addition, preserves the orientation and the time-orientation.
Let $\mathcal{C}^{\uparrow}_+(\mathbb{M})$ denote the group of all restricted conformal transformations
of $(\mathbb M, [g])$. We call $\mathcal{C}^{\uparrow}_+(\mathbb{M})$
the {\em restricted conformal group} of $\mathbb M$.
It is a classical result that $\mathcal{C}^{\uparrow}_+(\mathbb{M})$ is a Lie transformation
group.

\end{defn}




Let $\mathcal{A}=(A_1,\dots,A_{n+1})$ be a positive basis of the tangent space $T_p(\mathbb{M})$ at
a point $p\in \mathbb M$ and let $(A^1,\dots,A^{n+1})$
be its dual basis. We say that $\mathcal{A}$ is a {\it positive linear conformal frame} at $p\in \mathbb M$ if
$$
  -A^1\odot A^1+\sum_{j=2}^{n+1}A^j\odot A^{j}\in [g]_{|_p}, 
    \quad A_1
\in \dot{\mathrm{N}}^{\uparrow}(\mathbb{M}))_{|_p}.
    $$
The set $\mathrm{C}\O(\mathbb{M})$ of all positive linear conformal frames at all points of $\mathbb M$
defines a principal fiber bundle over $\mathbb{M}$, $\pi_{\mathrm{C}}:\mathrm{C}\O(\mathbb{M})\to \mathbb{M}$,
with structure group
\begin{equation}\label{s:2:s:1:f1}
 \mathrm{C}\O(1,n)=\left\{r\mathrm{B} \mid r>0, \,\mathrm{B}\in \O(1,n)\right\}.
  \end{equation}
As a subbundle of the bundle $\mathrm L(\mathbb M)$ of linear frames over $\mathbb M$, $\mathrm{C}\O(\mathbb{M})$
defines a $\mathrm{C}\O(1,n)$-structure on $\mathbb M$.
By a {\it positive linear conformal frame field} is meant a local section of $\pi_{\mathrm{C}}$.
Let $\mathcal{CO}^{\uparrow}_+ = \{\mathcal{A}^{\alpha}\}_{\alpha\in \mathcal{C}}$ denote
the sheaf of positive linear conformal frame fields.
For each $\alpha\in \mathcal{C}$, let $U_{\alpha}\subset \mathbb{M}$ be the domain of definition of $\mathcal{A}^{\alpha}$.
Let $\mathcal{C}_{2}=\{(\alpha,\beta)\in \mathcal{C}\times \mathcal{C}\, | \, U_{\alpha}\cap U_{\beta}\neq \emptyset\}$,
and let
 $\Gamma=\{\mathrm{A}_{\alpha}^{\beta}\}_{(\alpha,\beta)\in \mathcal{C}_2}$ denote the
{\u{C}}ech 1-cocycle
defined by the transition functions of $\mathrm{C}\O(\mathbb{M})$.
For each $\mathcal{A}^{\alpha}\in \mathcal{CO}^{\uparrow}_+$, let
$\omega_{\alpha}= {}^t\!(\omega^1_{\alpha},\dots,\omega^{n+1}_{\alpha})$ be the corresponding dual coframe.
Then, if $(\alpha,\beta)\in \mathcal{C}_2$,
\begin{equation}\label{s:2:s:2:f1}
  \mathcal{A}^{\beta} = \frac{1}{r_{\alpha}^{\beta}}\mathcal{A}^{\alpha}  \mathrm{B}^{\beta}_{\alpha},
   \end{equation}
where $r_{\alpha}^{\beta} : U_{\alpha}\cap U_{\beta}\to \R^{+}$ and
$\mathrm{B}_{\alpha}^{\beta}:U_{\alpha}\cap U_{\beta}\to \O(1,n)$.\footnote{Here $\R^+ = \{r\in \R \,|\, r>0\}$.}
Define
\begin{equation}\label{s:2:s:2:f2}
 y_{\alpha}^{\beta}= {}^t\!({y}_{\alpha}^{\beta, 1},{y}_{\alpha }^{\beta, 2},
   \dots,{y}_{\alpha}^{\beta, n+1}): U_{\alpha}\cap U_{\beta}\to \R^{1,n}
    \end{equation}
by requiring that
\[
  d \log r_{\alpha}^{\beta} =(\omega^1_{\alpha},\dots,\omega^{n+1}_{\alpha}) y_{\alpha}^{\beta}
  \]
and put
\[
{}^*\!y_{\alpha}^{\beta}=(-y_{\alpha}^{\beta, 1},
  y_{\alpha }^{\beta, 2},\dots,y_{\alpha}^{\beta, n+1}).
    \]
Let
\[
   \dot{\mathrm{A}}_{\alpha}^{\beta}:U_{\alpha}\cap U_{\beta}\to \mathrm{H}^{\uparrow}_+(2,n+1)
     \]
be the smooth $\mathrm{H}^{\uparrow}_+(2,n+1)$-valued map defined by
\[
  \dot{\mathrm{A}}_{\alpha}^{\beta}=\left(
                     \begin{array}{ccc}
                       r_{\alpha}^{\beta} & {}^*\!y_{\alpha}^{\beta} \mathrm{B}^{\beta}_{\alpha}  & {^*\!y_{\alpha}^{\beta} y_{\alpha}^{\beta}}/{2r_{\alpha}^{\beta}}\\
                       0 & \mathrm{B}^{\beta}_{\alpha} & {y_{\alpha}^{\beta}}/{r_{\alpha}^{\beta}} \\
                       0 & 0 & {1}/{r_{\alpha}^{\beta}} \\
                     \end{array}
                   \right).
                   \]
It is now a computational matter to check that $\dot{\Gamma}=\{\dot{\mathrm{A}}_{\alpha}^{\beta}\}_{(\alpha,\beta)\in \mathcal{C}_2}$
defines a {\u{C}}ech 1-cocycle with values in $\mathrm{H}^{\uparrow}_+(2,n+1)$ on the given covering of $\mathbb M$.
Consequently, see for instance \cite[Proposition 5.2, page 52]{KN},
there exists a unique principal fiber bundle $\pi_{\mathrm{Q}}:\mathrm{Q}(\mathbb{M})\to \mathbb{M}$,
with structure group $\mathrm{H}^{\uparrow}_+(2,n+1)$,
admitting
$\dot{\Gamma}$ as a 1-cocycle of transition functions
and with an atlas $\dot{\mathcal{CO}}^{\uparrow}_+=\{\dot{\mathcal{A}}^{\alpha}\}_{\alpha\in \mathcal{C}}$
of local sections.


\begin{defn}
We call $\mathrm{Q}(\mathbb{M})$ the (restricted) {\it Cartan conformal bundle} of $\mathbb{M}$.
\end{defn}

\begin{remark}
The construction of $\mathrm{Q}(\mathbb{M})$ from
$\mathrm{C}\O(\mathbb{M})$ is a particular instance
of the procedure known as prolongation of a $G$-structure (see \cite{K1}).
%
Starting from a $G$-structure $P\subset \mathrm L(\mathbb M)$, there is a canonical way to construct
a $G^{1}$-structure $P^1\subset \mathrm L(P)$. This construction can be repeated on $P^1$ to
obtain a $G^2$-structure $P^2$, and so on.
%
%
If $G\subset\GL(M,\R)$
is a Lie subgroup of finite order, say $k$,
then $G^{k} = \{e\}$, and the $k$th prolongation $P^k$ is an $\{e\}$-structure.
The bundle $P_{k-1}$ encodes all relevant pieces of information about the local geometry of the $G$-structure.
Usually, geometries of infinite order (e.g., complex, contact or
symplectic geometries) do not have local invariants. In the case at hand,
$\mathrm{Q}(\mathbb{M})=\mathrm{C}\O(\mathbb{M})^1$
and $\mathrm{Q}(\mathbb{M})^1$ is an $\{e\}$-structure on $\mathrm{Q}(\mathbb{M})$,
since the group $\mathrm{C}\O(1,n)$ has order 2.
\end{remark}

\subsection{The normal conformal connection}\label{s:2:s:2}

Let $\mathcal{A}^{\alpha}$ be a positive linear conformal frame field of $\mathbb{M}$
defined on $U_\alpha$
and
$\omega_{\alpha}= {}^t\!(\omega^1_{\alpha},\dots,\omega^{n+1}_{\alpha})$ the corresponding dual coframe field.
Then the quadratic form
\[
 \ell_{\alpha}=-\omega_{\alpha}^1\odot \omega_{\alpha}^1 +\sum_{j=2}^{n+1}\omega_{\alpha}^j\odot \omega_{\alpha}^j=\sum_{i,j=1}^{n+1}\widetilde{\delta}_{ij}\omega_{\alpha}^i\odot \omega_{\alpha}^j,
\quad \widetilde{\delta}_{ij}=\widetilde{\delta}_{ji},\, 1\le i,j\ \le n+1
  \]
  belongs to the conformal class $[g]$ on $U_\alpha$.
Consequently, there exists a unique $\mathfrak{o}(1,n)$-valued exterior differential 1-form
$\theta_{\alpha}= ({\theta_\alpha}^i_j)
\in \Omega^1(U_{\alpha})\otimes \mathfrak{o}(1,n)$,
the {\it Levi-Civita connection form} of $\ell_{\alpha}$
with respect to the pseudo-orthogonal frame field $\mathcal{A}^{\alpha}$, such that
\[
  d\omega_{\alpha}=-\theta_{\alpha}\wedge \omega_{\alpha}.
  \]
Consider the curvature form $\Theta_{\alpha}\in \Omega^2(U_{\alpha})\otimes \mathfrak{o}(1,n)$ defined by
\[
  \Theta_{\alpha} = d\theta_{\alpha}+\theta_{\alpha}\wedge \theta_{\alpha}.
   \]
We write
\[
 \Theta_{\alpha j}^{\hskip0.2cm i}= \frac{1}{2}\sum_{h,k=1}^{n+1}{(\mathrm{R}_{\alpha})}^i_{jhk}
\omega_{\alpha}^h\wedge \omega_{\alpha}^k,
   \]
where the functions ${(\mathrm{R}_{\alpha})}^i_{jhk}$
are the {\it local components of the Riemann curvature tensor} of $\ell_{\alpha}$ with respect to $\mathcal{A}^{\alpha}$.
The {\it Ricci tensor components} ${(\mathrm{R}_\alpha)}_{jh}$ and the {\it scalar curvature} $\mathrm R_{\alpha}$ are
given by
\[
  {(\mathrm{R}_{\alpha})}_{jh} = \sum_{k=1}^{n+1}{(\mathrm{R}_{\alpha})}^k_{jhk},
\quad \mathrm R_{\alpha}= \sum_{h=1}^{n+1} {(\mathrm{R}_{\alpha})}_{hh}.
    \]
 Let $\eta_{\alpha}=(\eta_{\alpha,1},\dots,\eta_{\alpha,n+1})\in \Omega^1(U_{\alpha})\otimes (\R^{1,n})^\ast$ be
 the vector-valued 1-form defined by
\[
   \eta_{\alpha,j}= \frac{\mathrm{R}_{\alpha}}{2n(n-1)}\sum_{h=1}^{n+1}\widetilde{\delta}_{jh}\omega^h_{\alpha}-
 \frac{1}{n-1}\sum_{h=1}^{n+1}{(\mathrm{R}_{\alpha})}_{jh}\omega^h_{\alpha},\quad j = 1,\dots,n+1.
    \]
Next, let
\[
{}^*\!\omega_{\alpha} = (-\omega^1_{\alpha},\omega^2_{\alpha},\dots,\omega^{n+1}_{\alpha}),
 \quad {}^*\!\eta_{\alpha}= {}^t\!(-\eta_{\alpha, 1},\dots,\eta_{\alpha, n+1})
    \]
 and consider the $\mathfrak{m}(2,n+1)$-valued 1-form $\phi_{\alpha}$ given by
\[
    \phi_{\alpha}=\left(
                   \begin{array}{ccc}
                     0 & \eta_{\alpha} & 0 \\
                     \omega_{\alpha} & \theta_{\alpha} & {}^*\!\eta_{\alpha} \\
                     0 & {}^*\!\omega_{\alpha} & 0 \\
                   \end{array}
                 \right).
                   \]
By similar calculations as those performed in the Riemannian case in \cite[Chapter 1, Section 4]{Ya}
for determining
the transformation rules, under a conformal metric change, of the components
of the Levi-Civita connection
and of the modified Ricci tensor in the expression of $\eta_{\alpha,j}$,
it can be verified that if $\mathcal{A}^{\alpha}=\mathcal{A}^{\beta}  \mathrm{A}^{\alpha}_{\beta}$, then
\begin{equation}\label{gtl}
  \phi_{\alpha} =
(\dot{\mathrm{A}}^{\alpha}_{\beta})^{-1}  \phi_{\beta} \,\dot{\mathrm{A}}^{\alpha}_{\beta}+ (\dot{\mathrm{A}}^{\alpha}_{\beta})^{-1}d\dot{\mathrm{A}}^{\alpha}_{\beta}.
     \end{equation}

According to \cite[Proposition 1.4, page 66]{KN}, we can thus state the following.
%

\begin{prop}
There exists a unique exterior differential 1-form $\phi\in \Omega^1[\mathrm{Q}(\mathbb{M})]$ $\otimes$ $\mathfrak{m}(2,n+1)$
such that:
\begin{itemize}
\item if $\xi\in \mathfrak{h}(2,n+1)$, then $\phi(\xi^*)=\xi$;\footnote{Here $\xi^*$
stands for the fundamental vector field on $\mathrm{Q}(\mathbb{M})$ generated by $\xi$.}
\item $R^*_{\mathrm{X}}(\phi)=\mathrm{X}^{-1}  \phi \mathrm{X}$, for every $\mathrm{X}\in \mathrm{H}^{\uparrow}_+(2,n+1)$;
\item $\dot{\mathcal{A}}_{\alpha}^*(\phi)=\phi_{\alpha}$, $\forall \dot{\mathcal{A}}_{\alpha}\in \dot{\mathcal{CO}}^{\uparrow}_+$.
\end{itemize}
\end{prop}

\begin{defn}
The 1-form $\phi$ is the {\it normal conformal Cartan connection} of the conformal Lorentz manifold $\mathbb{M}$.
\end{defn}

\begin{remark}\label{r:conf-prolong}
Using block matrix notation, the normal connection takes the form
\[
\phi = \left(
           \begin{array}{ccc}
             \phi_0^0 & \phi^0 & 0 \\
             \phi_0 & \widehat{\phi} & {}^*\!\phi^0 \\
             0 & {}^t\!\phi_0 & -  \phi_0^0 \\  
           \end{array}
         \right),
      \]
where $\phi^0_0$ is a scalar 1-form, $\phi_0 ={}^t\!(\phi^1_0, \dots,\phi^{n+1}_0)$ and
$\phi^0 = (\phi^0_1,\dots, \phi^0_{n+1})$ are vector-valued 1-forms, and $\widehat{\phi} = (\phi^i_j)$, $i,j =1,\dots,n+1$,
is a 1-form taking values in the Lie algebra $\mathfrak{o}(1,n)$.
The connection form $\phi$ satisfies the following basic properties.

\begin{enumerate}

\item The 1-forms $\phi^0_0$, $\phi^1_0,\dots,\phi^{n+1}_0$, $\phi^0_1,\dots,\phi^0_{n+1}$,
$\phi^i_j$, $1\leq i<j\leq n+1$, are linearly independent and define an absolute parallelism
on $\mathrm{Q}(\mathbb{M})$, the {\it normal parallelism} of the Cartan conformal bundle.

\item The 1-forms
$\phi^1_0, \dots,\phi^{n+1}_0$ are semibasic.\footnote{A form is semibasic of it
annihilates the vertical vector fields.}

\item For every
restricted
conformal transformation $F:\mathbb{M}\to \mathbb{M}$,
 there exists a unique lifted diffeomorphism $\dot{F}:\mathrm Q(\mathbb{M})\to \mathrm Q(\mathbb{M})$,
 called the {\it conformal prolongation of $F$}, that preserves the normal conformal connection $\phi$, i.e.,
$(\dot{F})^{*}(\phi)=\phi$.
Conversely, any diffeomorphism $\mathcal F$ of $\mathrm Q(\mathbb{M})$ that preserves $\phi$ arises in this way, that is,
there exists a unique restricted conformal transformation $F$
of $\mathbb M$ such that $\mathcal F = \dot F$.
In particular, if $\mathrm{Aut} (\phi)$ denotes the group of diffeomorphisms of $\mathrm Q(\mathbb{M})$ preserving $\phi$,
then the mapping $F\in \mathcal{C}^{\uparrow}_+(\mathbb{M})\mapsto \dot F \in \mathrm{Aut} (\phi)$ is an
isomorphism.
\end{enumerate}
\end{remark}

\subsection{The restricted conformal group of a Lorentzian manifold}\label{s:2:s:3}

We are now in a position to apply
to $\mathrm Q(\mathbb{M})$, with the normal parallelism induced by
the normal conformal connection,
a classical result of S. Kobayashi on the transformation group of a manifold with an absolute
parallelism \cite[Chapter I, Theorem 3.2]{K1}.

\begin{thm}[S. Kobayashi]
Let $N$ be an $n$-dimensional differentiable manifold with an absolute parallelism, i.e., a
coframe $\{\eta^1, \dots, \eta^n\}$ of globally defined 1-forms which are linearly
independent at each point of $N$. Let $G = \mathrm{Aut}(\{\eta^i\})$ be the group of automorphisms of the absolute parallelism,
i.e., the group
of diffeomorphisms $F:N \to N$, such that $F^\ast(\eta^i) = \eta^i$, $i = 1, \dots, n$.
Then, $G$ is a Lie transformation group such that $\mathrm{dim}\,\mathrm{G}$ $\leq$ $n$. More precisely,
for any $p\in N$, the mapping $F\in G \mapsto F(p) \in N$ is injective and its image
$\{F(p) \mid F\in G\}$ is a closed submanifold of $N$. The submanifold structure on this image makes $G$
into a Lie transformation group.

\end{thm}

Accordingly, taking into account the above discussion, we can state the following.

\begin{thm}\label{THM0}
Let $\mathbb{M}$ be an oriented, time-oriented conformal Lorentz manifold of dimension $n+1\ge 3$.
Let ${L}:\mathrm{G}\times \mathbb{M}\to \mathbb{M}$ be an effective left action of a connected
Lie group $\mathrm{G}$ on
$\mathbb{M}$ by restricted conformal transformations.
For any $\dot{\mathcal{A}}_*\in \mathrm Q(\mathbb{M})$, the mapping
\[
  \mathfrak{j}: \mathrm{G} \ni g \longmapsto \dot{{L}}_g(\dot{\mathcal{A}}_*)\in \mathrm Q(\mathbb{M})
  \]
is a 1-1 immersion of $\mathrm{G}$ into $\mathrm Q(\mathbb{M})$.
Here, for each $g\in \mathrm G$, $\dot{{L}}_g$ denotes the conformal prolongation of $L_g$ (cf. Remark
\ref{r:conf-prolong}).
In particular,
$\mathrm{dim}\,\mathrm{G}$ $\leq$ $\mathrm{dim}\, \mathrm Q(\mathbb{M})$ $=$ $\frac12(n+3)(n+2)$.
If $\mathrm{dim}\,\mathrm{G} = \frac12(n+3)(n+2)$, then:

\begin{enumerate}

\item  The mapping $\mathfrak{j}$ is a diffeomorphism and $\mathrm Q(\mathbb{M})$ inherits from $\mathrm{G}$ the structure of a Lie group with neutral element $\dot{\mathcal{A}}_*$.
    Moreover, $\mathrm G$ and $\mathrm Q(\mathbb{M})$ are isomorphic to $\mathcal{C}^{\uparrow}_+(\mathbb{M})$,
    the restricted conformal group of $\mathbb M$.

\item  $\mathrm Q(\mathbb{M})$ is locally isomorphic to $\mathrm{M}^{\uparrow}_+(2,n+1)$.

\item If $p_*=\pi_{\mathrm Q}(\dot{\mathcal{\mathcal{A}}}_*)$, then the fiber $\pi_{\mathrm{Q}}^{-1}(p_*)$ is a connected Lie subgroup isomorphic to $\mathrm{H}_+^{\uparrow}(2,n+1)$.

\item  $\pi_{\mathrm{Q}}^{-1}(p_*)$ is the stabilizer of the point $p_*$ for the action of $\mathrm Q(\mathbb{M})$ on $\mathbb{M}$.

\item  $\pi_{\mathrm{Q}}^{-1}(p_*)$ is the maximal integral submanifold through $\dot{\mathcal{A}}_*$ of the left-invariant, completely integrable Pfaffian differential system generated by the 1-forms $\phi^1_0,\dots,\phi^{n+1}_0$.

\end{enumerate}

\end{thm}

\begin{ex}\label{cpt-EU}

Let $\mathcal{N}^+_1(\R^{2,n+1})$ and $\mathcal{N}_1(\R^{2,n+1})$ denote, respectively,
the manifolds of all oriented and unoriented isotropic lines
through the origin in $\R^{2,n+1}$.
%
%
For a nonzero isotropic vector
$x\in \R^{2,n+1}$, we let $|[x]|$ and $[x]$
denote, respectively, the oriented and the unoriented line spanned by
$x$.
{Consider the smooth maps}
\[
\begin{cases}
 \E \ni (\tau,\mathrm{y}) \longmapsto |[(\rho_1(\tau),\mathrm{y})]|\in \mathcal{N}^+_1(\R^{2,n+1}),\\
 \E \ni (\tau,\mathrm{y}) \longmapsto [(\rho_1(\tau),\mathrm{y})]\in \mathcal{N}_1(\R^{2,n+1}),
 \end{cases}
   \]
{where $\rho_1(\tau) = {}^t\!(\cos \tau, \sin \tau)$}.
The first map is invariant under the action of $\mathcal{T}_{2\pi}$, while the second is invariant under
the action of $\mathcal{T}'_{\pi}$. They induce smooth diffeomorphisms
\[
 \widehat{\eta}_{\mathrm{I}}:\EC \to \mathcal{N}^+_1(\R^{2,n+1}),\quad
  \widehat{\eta}_{\mathrm{II}}:\ECC \to \mathcal{N}_1(\R^{2,n+1}).
  \]
{Since for $n$ even, $\ECC$ is not orientable, we will only consider $\ECC$ when $n$ is odd.}
With the above identifications, we get smooth left actions of $\O(2,n+1)$
and $\mathrm{P}\O(2,n+1):=\O(2,n+1)/\{\pm I\}$ on the standard compact forms. These actions are conformal
and preserve orientation and time-orientation. For dimensional reasons,
$\O(2,n+1)$ and $\mathrm{P}\O(2,n+1)$ are isomorphic to the restricted conformal groups of $\EC$ and $\ECC$, respectively.
Notice that the Maurer--Cartan forms of the groups coincide with the normal conformal connection forms.
\end{ex}

\section{The canonical covering of $\O(2,n+1)$}\label{s:3}

In this section we will build a nontrivial central extension $\OO(2,n+1)$ of the pseudo-orthogonal
group $\O(2,n+1)$ and will describe its center $\widehat{\mathrm{Z}}(2,n+1)$.


\subsection{Transitive action of $\O(2,n+1)$ on the classical domain of type IV}



\begin{defn}
Let $\Omega_{\mathrm{IV}}$
denote the set of $(n+1)\times 2$ real matrices defined by
\[
\Omega_{\mathrm{IV}} := \left\{ \beta \in \R({n+1,2})
\mid  I_2 - {}^t\!\beta\beta >0,
   \,\,\text{i.e., $I_2 -{}^t\!\beta\beta$ is positive definite} \right\}.
   \]
 \end{defn}
%
%
%
The set $\Omega_{\mathrm{IV}}$ can be thought of as the open domain of $\R(n+1,2)$ consisting of all
$(n+1)\times 2$ matrices $\beta = (\mathrm{u}\,\mathrm{v})$ whose column vectors $\mathrm{u},
\mathrm{v} \in \R^{n+1}$ satisfy\footnote{For $\mathrm{u},
\mathrm{v} \in \R^{n+1}$, $\mathrm{u}\cdot\mathrm{v}$ denotes the usual dot product
and $\|\mathrm{u}\|$ the corresponding norm.}
\[
 \|\mathrm{u}\|<1, \quad  \|\mathrm{v}\|<1, \quad
 \mu(\mathrm{u},\mathrm{v}):=\|\mathrm{u}\|^2+\|\mathrm{v}\|^2
  +(\mathrm{u}\cdot \mathrm{v})^2-\|\mathrm{u}\|^2\|\mathrm{v}\|^2 <1,
  \]
that is,
\[
  \Omega_{\mathrm{IV}}=\left\{\beta = (\mathrm{u} \,\mathrm{v}) \in \mathbb \R({n+1,2})
   \mid \|\mathrm{u}\|^2<1,\, \|\mathrm{v}\|^2<1,\, \mu(\mathrm{u},\mathrm{v})<1\right\}.
    \]

\begin{remark}
Let $\mathfrak D$ denote the bounded domain of $\mathbb C^{n+1}$ defined by
\[
 \mathfrak D := \left\{ z \in \mathbb C^{n+1} \mid  2\,{}^t\!z \bar z< 1 + |{}^t\!zz|^2   < 2 \right\}.
   \]
The domain $\mathfrak D$ is known in the literature as the (complex) {\it Lie ball}.
If $n = 0$, it is the unit disk, $\mathbb D$, and if $n = 1$, $\mathfrak D\cong \mathbb D\times \mathbb D$.
According to Hua \cite[\S13]{Hua}, the Lie ball $\mathfrak D$ can be identified
with the matrix domain $\Omega_{\mathrm{IV}}$.
In fact,
the mapping $H : \mathfrak D  \to \Omega_{\mathrm{IV}}$, defined by
\[
H(z) : = 2\begin{pmatrix}
z & \bar z
\end{pmatrix}
\begin{pmatrix}
{}^t\!z z +1 & \overline{{}^t\!z z} +1\\
i({}^t\!z z -1) & -i(\overline{{}^t\!z z} -1)
\end{pmatrix}^{-1}
\]
is a diffeomorphism of $\mathfrak D$ onto $\Omega_{\mathrm{IV}}$.
%
Actually, if one identifies the matrix $\beta = (\mathrm u\,  \mathrm v)$ with the vector
$\mathrm u + i\mathrm v \in \mathbb C^{n+1}$,
in order to obtain an almost complex structure on $\Omega_{\mathrm{IV}}$, the map $H$ becomes a
holomorphic diffeomorphism.
The domain $\mathfrak D$ ($n\geq 2$) is a classical irreducible
bounded symmetric domain of type IV
(cf. \cite{BO, Ca1, Hel, Hua, HuaBook} for more details).
\end{remark}

\begin{notation}
Let $\mathrm{\mathbf{X}}\in \O(2,n+1)$ be written in block form as
\[
 \mathbf{X}=\begin{pmatrix}
                 \mathfrak{a}(\mathrm{\mathbf{X}}) & \mathfrak{b}(\mathrm{\mathbf{X}}) \\
                 \mathfrak{c}(\mathrm{\mathbf{X}}) & \mathfrak{d}(\mathrm{\mathbf{X}}) \\
\end{pmatrix},
   \]
where $\mathfrak{a}(\mathrm{\mathbf{X}})\in \R(2,2)$, $\mathfrak{b}(\mathrm{\mathbf{X}})\in \R(2,n+1)$, $\mathfrak{c}(\mathrm{\mathbf{X}})\in \R(n+1,2)$, $\mathfrak{d}(\mathrm{\mathbf{X}})$ $\in$ $\R(n+1,n+1)$
and
\[
\begin{pmatrix}
\mathfrak{a}(\mathrm{\mathbf{X}}) & \mathfrak{b}(\mathrm{\mathbf{X}}) \\
\mathfrak{c}(\mathrm{\mathbf{X}}) & \mathfrak{d}(\mathrm{\mathbf{X}}) \\
\end{pmatrix}
\begin{pmatrix}
I_2 & 0 \\
0 & -I_{n+1} \\
\end{pmatrix}
\begin{pmatrix}
{}^t\!\mathfrak{a}(\mathrm{\mathbf{X}}) & {}^t\!\mathfrak{c}(\mathrm{\mathbf{X}}) \\
{}^t\!\mathfrak{b}(\mathrm{\mathbf{X}}) & {}^t\!\mathfrak{d}(\mathrm{\mathbf{X}}) \\
\end{pmatrix} =
\begin{pmatrix}
I_2 & 0 \\
0 & -I_{n+1} \\
\end{pmatrix},
\]
that is,
\[
\begin{cases}
{}^t\!\mathfrak{a}(\mathrm{\mathbf{X}})\mathfrak{a}(\mathrm{\mathbf{X}})=I_{2} +
{}^t\!\mathfrak{c}(\mathrm{\mathbf{X}}) \mathfrak{c}(\mathrm{\mathbf{X}}),\\
{}^t\!\mathfrak{d}(\mathrm{\mathbf{X}})\mathfrak{d}(\mathrm{\mathbf{X}})=I_{n+1}+
{}^t\!\mathfrak{b}(\mathrm{\mathbf{X}}) \mathfrak{b}(\mathrm{\mathbf{X}}),\\
{}^t\!\mathfrak{a}(\mathrm{\mathbf{X}}) \mathfrak{b}(\mathrm{\mathbf{X}})- {}^t\!\mathfrak{c}(\mathrm{\mathbf{X}})
\mathfrak{d}(\mathrm{\mathbf{X}}) = 0.
 \end{cases}
  \]
Since $\mathbf X$ preserves time and space-orientation, the {\it timelike part} $\mathfrak{a}(\mathbf X)$ is invertible
with $\text{det}\, \mathfrak{a}(\mathbf X) >0$. Similarly, the {\it spacelike part}
$\mathfrak{d}(\mathbf X)$ is invertible
 and $\text{det} \,\mathfrak{d}(\mathbf X) >0$.
\end{notation}

With the above notation, it is easy to show that the group $\O(2,n+1)$ acts transitively
on $\Omega_{\mathrm{IV}}$ by
\[
 L_{\mathrm{\mathbf{X}}} (\beta) := \left(\mathfrak{d}(\mathrm{\mathbf{X}}) \beta
  +\mathfrak{c}(\mathrm{\mathbf{X}})\right) \left(\mathfrak{b}(\mathrm{\mathbf{X}})\beta
    + \mathfrak{a}(\mathrm{\mathbf{X}})\right)^{-1}.
    \]
For a proof of this fact we refer to \cite[\S12]{Hua}.
The isotropy group at the origin $\mathrm{O}_{\mathrm{IV}} := 0_{(n+1)\times 2} \in \Omega_{\mathrm{IV}}$ is
\[
 \mathrm{SO}(2)\times \mathrm{SO}(n+1)\cong \left\{\mathrm{\mathbf{S}}(\mathrm{r},\mathrm{{R}})=
                                                 \begin{pmatrix}
                                                   \mathrm{r} & 0 \\
                                                   0 & \mathrm{{R}} \\
                                                 \end{pmatrix}  \mid \mathrm{r}\in \mathrm{SO}(2),\,
                                               \mathrm{{R}}\in \mathrm{SO}(n+1)\right\}.
   \]
This gives a coset expression for $\Omega_{\mathrm{IV}}$,
\[
 \Omega_{\mathrm{IV}} =\O(2,n+1)/\mathrm{SO}(2)\times \mathrm{SO}(n+1).
  \]
The canonical projection of $\O(2,n+1)$ onto $\Omega_{\mathrm{IV}}$,
\[
\pi_2^- : \O(2,n+1)\to  \Omega_{\mathrm{IV}}, \,
  \mathbf X \longmapsto L_{\mathrm{\mathbf{X}}}(\mathrm{O}_{\mathrm{IV}}) =
 \mathfrak{c}(\mathrm{\mathbf{X}})  \mathfrak{a}(\mathrm{\mathbf{X}})^{-1},
 \]
makes $\O(2,n+1)$ into a principal bundle
over $\Omega_{\mathrm{IV}}$ with group $\mathrm{SO}(2)\times \mathrm{SO}(n+1)$.


\vskip0.2cm
The following is a well-known property of
classical domains \cite{Hel, HuaBook}.
For future use, we state it explicitly for $\Omega_{\mathrm{IV}}$.

\begin{lemma}
 The domain $\Omega_{\mathrm{IV}}$ is star-shaped with respect to the origin
  $\mathrm{O}_{\mathrm{IV}}=0_{(n+1)\times 2}$. In particular, $\Omega_{\mathrm{IV}}$ is contractible.
  \end{lemma}


\begin{remark}
The above property of $\Omega_{\mathrm{IV}}$ also follows from the general approach to irreducible
bounded symmetric domains
based on the theory of Riemannian symmetric spaces.
In fact, the domain $\Omega_{\mathrm{IV}}$ may be given the structure of a Riemannian
symmetric space of noncompact type. As such, it has non-positive sectional curvature
and negative definite Ricci tensor,
which implies that
 $\Omega_{\mathrm{IV}}$ is simply connected, and hence diffeomorphic to Euclidean space $\R^{2(n+1)}$
(cf. \cite{Hel, ON}).
\end{remark}


\begin{defn}
Let $\NG2$ denote the Grassmannian of
{\it negative} 2-planes, i.e.,
%
the Grassmannian of 2-dimensional subspaces $\V$ of $\R^{2,n+1}$ on which the scalar product $\langle \,, \rangle$
is negative definite.
 \end{defn}

The matrix domain $\Omega_{\mathrm{IV}}$ can be identified with
$\NG2$ by the
 mapping
\begin{equation}\label{ident-R-NG2}
 \mathbf j : \Omega_{\mathrm{IV}} \to \NG2, \quad \mathbf j(\beta) :=
  \text{span}\left\{\mathbf{j}_1(\beta),\mathbf{j}_2(\beta) \right\},
%
  \end{equation}
where, for each $\beta=(\mathrm{u}\,\mathrm{v})\in \Omega_{\mathrm{IV}}$,
$\mathbf{j}_1(\beta)={}^t\!(1, 0 , {}^t\!\mathrm u)$
and $\mathbf{j}_2(\beta)= {}^t\!( 0, 1, {}^t\!\mathrm{v})$.
In fact, we have the following.
%

%
%

\begin{lemma}
The mapping $\mathbf{j}: \Omega_{\mathrm{IV}}\to \NG2$
is a smooth diffeomorphism.
\end{lemma}

\begin{proof}
By construction, $\mathbf{j}$ is injective, differentiable and of maximal rank. Thus it suffices to show that
$\mathbf{j}$ is surjective. Let $\mathbb{V}\in \NG2$ be a negative 2-plane and let
$(\mathrm{v}_1,\mathrm{v}_2)$ be an orthogonal basis of $\mathbb{V}$.
Let us write ${}^t\!\mathrm{v}_1 =({}^t\!\mathrm{x}, {}^t\!\mathrm{y})$ and
${}^t\!\mathrm{v}_2= ({}^t\!\mathrm{x}', {}^t\!\mathrm{y}')$,
where $\mathrm{x}$, $\mathrm{x}'\in \R^2$
and $\mathrm{y}$, $\mathrm{y}'\in \R^{n+1}$.
The vectors $\mathrm{x},\mathrm{x}'$ and $\mathrm{y},\mathrm{y}'$ satisfy the identities
\begin{equation}\label{cnp}
-\|\mathrm{x}\|^2+\|\mathrm{y}\|^2=-\|\mathrm{x}'\|^2+
  \|\mathrm{y}'\|^2=-1,\quad -\mathrm{x}\cdot \mathrm{x}'+\mathrm{y}\cdot \mathrm{y}'=0.
   \end{equation}
First, observe that $\mathrm{x}$ and $\mathrm{x}'$ must be different from zero. In fact,
if $\mathrm{x}$ or $\mathrm{x}'$ were zero, the first equation in \eqref{cnp}
would imply $0\le \|\mathrm{y}\|=-1$ or $0\le \|\mathrm{y}'\|=-1$, which is a contradiction.
Next, we claim that $\mathrm{x}$ and $\mathrm{x}'$ are linearly independent.
%
%
%
Seeking a contradiction, suppose that $\mathrm{x}'=t\mathrm{x}$, for a nonzero real number $t$.
If we write
 $\mathrm{y}'\cdot\mathrm{y}=\|\mathrm{y}'\| \|\mathrm{y}\|\cos \theta $, where $\theta\in [0,\pi]$,
the identities (\ref{cnp}) can be rewritten as
\begin{equation}\label{cnpbis}
  \|\mathrm{y}\|^2=\|\mathrm{x}\|^2-1,\quad \|\mathrm{y}'\|^2=t^2\|\mathrm{x}\|^2-1,\quad
   t\|\mathrm{x}\|^2=\|\mathrm{y}\|\|\mathrm{y}'\|\cos \theta,
    \end{equation}
which implies
\[
  \begin{split}
  0&=t^2\|\mathrm{x}\|^4-\|\mathrm{y}\|^2\|\mathrm{y}'\|^2 \cos^2\theta\\
&=t^2\|\mathrm{x}\|^4 - \left(\|\mathrm{x}\|^2-1\right)\left(t^2\|\mathrm{x}\|^2-1\right)\cos^2\theta\\
&=t^2\|\mathrm{x}\|^4 \left(1-\cos^2\theta\right) + \left(\|\mathrm{x}\|^2(1+t^2)-1\right)\cos^2\theta.
   \end{split}
    \]
By the first condition in \eqref{cnpbis}, it follows that $\|\mathrm{x}\|> 1$.
Taking this into account,
the previous equation implies $\|\mathrm{x}\|^2(1+t^2)-1=0$, which is the desired contradiction.
%
%
Now, since $\mathrm{x}$ and $\mathrm{x}'$ are linearly independent, the $2\times 2$ matrix $(\mathrm{x}\, \mathrm{x}')$
is invertible and
$\beta = (\mathrm{y}\, \mathrm{y}')(\mathrm{x} \,\mathrm{x}')^{-1}$ is an element of $\Omega_{\mathrm{IV}}$
such that $\mathbf{j}(\beta)=\mathbb{V}$.
\end{proof}

\subsection{Construction of the canonical covering}

For each $\beta = (\mathrm{u}\,\mathrm{v})\in \Omega_{\mathrm{IV}}$, where $\mathrm u={}^t\!(u^1,\dots,u^{n+1})$,
$\mathrm v={}^t\!(v^1,\dots,v^{n+1})$, we let\footnote{Here $\delta_{j}^{i}$ denotes the Kronecker symbol.}
\[
  \mathrm{B}_j(\beta):={}^t\!(u_j,v_j,\delta_{j}^{1},\dots,\delta_{j}^{n+1}),\quad j=1,\dots,n+1,
     \]
and, as above,
\[
  \mathbf{j}_1(\beta) = {}^t\!(1,0,u^1,\dots,u^{n+1}),\quad
  \mathbf{j}_2(\beta) = {}^t\!(0,1,v^1,\dots,v^{n+1}).
  \]
Then
\[
 \mathrm{B}(\beta) =\left(\mathbf{j}_1(\beta),\mathbf{j}_2(\beta),
  \mathrm{B}_1(\beta),\dots,\mathrm{B}_{n+1}(\beta) \right)
   \]
is a
positive basis of $\R^{2,n+1}$, such that

\begin{itemize}
\item $ \mathbf{j}_1(\beta)$, $\mathbf{j}_2(\beta)$ span the negative 2-space $\mathbf{j}(\beta)$;

\item $\mathrm{B}_1(\beta)$, $\dots$, $\mathrm{B}_{n+1}(\beta)$ span the positive
$(n+1)$-space $\mathbf{j}(\beta)^{\perp}$
of $\R^{2,n+1}$.
\end{itemize}

Consequently, by the Gram--Schmidt process,
there is a unique smooth map
$\mathfrak{T}:\Omega_{\mathrm{IV}}\to \mathrm{T}^+(n+3)$
into the group of upper triangular $(n+3)\times (n+3)$ matrices
with positive entries on the main diagonal
such that,
for each $\beta \in \Omega_{\mathrm{IV}}$, $\mathrm{P}(\beta):=\mathrm{B}(\beta)\mathfrak{T}(\beta)$ belongs to
$\O(2,n+1)$.
The map
\[
  \mathrm{P}: \Omega_{\mathrm{IV}}\ni \beta
\longmapsto
\mathrm{P}(\beta)\in \O(2,n+1)
   \]
is a smooth global cross section of $\pi^-_2:\O(2,n+1)\to \Omega_{\mathrm{IV}}$. Let
\[
 \begin{split}\widehat{\mathfrak{a}}&: \Omega_{\mathrm{IV}}\to \mathrm{GL}_+(2,\R),\quad \widehat{\mathfrak{d}}: \Omega_{\mathrm{IV}}\to \mathrm{GL}_+(n+1,\R),\\
\widehat{\mathfrak{c}}&: \Omega_{\mathrm{IV}}\to \R(n+1,2),\quad
\widehat{\mathfrak{b}}: \Omega_{\mathrm{IV}}\to \R(2,n+1)
  \end{split}
   \]
be the smooth maps defined by
\[
 \mathrm{P}(\beta)=
\begin{pmatrix}
\widehat{\mathfrak{a}}(\beta) & \widehat{\mathfrak{b}}(\beta) \\
\widehat{\mathfrak{c}}(\beta) & \widehat{\mathfrak{d}}(\beta) \\
\end{pmatrix},
\quad \forall\, \beta\in \Omega_{\mathrm{IV}}.
\]
Moreover, for each $t\in \R$, let
\[
 \mathbf{\rho}(t) = (\rho_1(t),\rho_2(t))=
 \begin{pmatrix}
   \cos t & -\sin t \\
 \sin t & \quad \cos t  \\
\end{pmatrix}
 \in \mathrm{SO}(2).
\]

We have the following.

\begin{lemma}\label{lemmaa}
The map $\widehat{\mathfrak{a}}:\Omega_{\mathrm{IV}}\to \mathrm{GL}_+(2,\R)$ has the
following invariance properties:
\[
   \widehat{\mathfrak{a}}(\mathrm{R}  \beta) =
    \widehat{\mathfrak{a}}(\beta), \quad
      {\widehat{\mathfrak{a}}(\beta \mathrm{r}^{-1})}^{-1}\,
        \mathrm{r} \, \widehat{\mathfrak{a}}(\beta) \in \mathrm{SO}(2),
   \]
for each $\beta \in \Omega_{\mathrm{IV}}$, $\mathrm{R}\in \mathrm{SO}(n+1)$, and $\mathrm{r}\in \mathrm{SO}(2)$.
In addition, there exists a unique smooth map $\eta : \Omega_{\mathrm{IV}}\times \mathrm{SO}(2)\to \R$, such that
\[
 \rho\left(\eta(\beta,\mathrm{r})\right)= {\widehat{\mathfrak{a}}(\beta\,  \mathrm{r}^{-1})}^{-1}\, \mathrm{r}\, \widehat{\mathfrak{a}}(\beta)\, \mathrm{r}^{-1},\quad
\eta(\mathrm{O}_{\mathrm{IV}},I_{2})=0,
        \]
for each $\beta\in \Omega_{\mathrm{IV}}$ and $\mathrm{r}\in \mathrm{SO}(2)$.
\end{lemma}

\begin{proof}
From the definition of $\mathrm{P} : \Omega_{\mathrm{IV}}\to \O(2,n+1)$,
it follows that
\begin{eqnarray}
 \widehat{\mathfrak{a}}(\beta) &=&
\frac{1}{\sqrt{1-\|\mathrm{u}\|^2}}
\begin{pmatrix}
1 & \frac{\mathrm{u}\cdot\mathrm{v}}{\sqrt{1-\mu(\mathrm{u},\mathrm{v})}} \\
0 & \frac{1-\|\mathrm{u}\|^2}{\sqrt{1-\mu(\mathrm{u},\mathrm{v})}} \\
      \end{pmatrix},\label{a-beta-u-v} \\
  \widehat{\mathfrak{c}}(\beta)&=&
\begin{pmatrix}
\frac{\mathrm{u}}{\sqrt{1-\|\mathrm{u}\|^2}} &
 \frac{(1-\|\mathrm{u}\|^2) \mathrm{v}
+ (\mathrm{u}\cdot \mathrm{v}) \mathrm{u}}{\sqrt{1 - \|\mathrm{u}\|^2}\sqrt{1-\mu(\mathrm{u},\mathrm{v})}} \\
      \end{pmatrix},\label{c-beta-u-v}
   \end{eqnarray}
for each $\beta = (\mathrm u\,\mathrm v) \in \Omega_{\mathrm{IV}}$.
This implies that
$\widehat{\mathfrak{a}}(\mathrm{R}\, \beta) = \widehat{\mathfrak{a}}(\beta)$,
for every
$\mathrm{R}\in \mathrm{SO}(n+1)$.
For $\mathrm{r}\in \mathrm{SO}(2)$ and $\beta\in \Omega_{\mathrm{IV}}$, we have
\begin{equation}\label{eqv1}
\mathrm{\mathbf{S}}(\mathrm{r}, {I}_{n+1})\, \mathrm{P}(\beta)=
 \begin{pmatrix}
  \mathrm{r}\, \widehat{\mathfrak{a}}(\beta)  & \mathrm{r}\, \widehat{\mathfrak{b}}(\beta) \\
   \widehat{\mathfrak{c}}(\beta) & \widehat{\mathfrak{d}}(\beta) \\
    \end{pmatrix}.
      \end{equation}
A direct computation taking into account \eqref{a-beta-u-v} and \eqref{c-beta-u-v} shows that
\[
\pi^-_2[\mathrm{\mathbf{S}}(\mathrm{r},I_{n+1})\, \mathrm{P}(\beta)]
=\widehat{\mathfrak{c}}(\beta)\, {\widehat{\mathfrak{a}}(\beta)}^{-1} \mathrm{r}^{-1}
 =\beta\, \mathrm{r}^{-1}.
  \]
This implies that the left hand side of \eqref{eqv1} belongs to the fiber
$(\pi_2^-)^{-1}(\beta\,\mathrm{r}^{-1})$.
As a consequence, we can write
\begin{equation}\label{eqv2}
 \mathrm{\mathbf{S}}(\mathrm{r},{I}_{n+1})\, \mathrm{P}(\beta)
 =\mathrm{P}(\beta\,\mathrm{r}^{-1})\,
  \begin{pmatrix}
   \widehat{ \mathrm{r}} & 0 \\
    0 & \widehat{\mathrm{R}} \\
  \end{pmatrix},
   \end{equation}
for $\widehat{\mathrm{r}}\in \mathrm{SO}(2)$ and $\widehat{\mathrm{R}}\in \mathrm{SO}(n+1)$.
Combining \eqref{eqv1} and \eqref{eqv2}, we conclude that
\[
{\widehat{\mathfrak{a}}(\beta\,\mathrm{r}^{-1})}^{-1}\, \mathrm{r}\, \widehat{\mathfrak{a}}(\beta)
=\widehat{ \mathrm{r}}\in \mathrm{SO}(2).
   \]

As for the existence of $\eta$,
consider the map
\[
 \widehat{\eta}: \Omega_{\mathrm{IV}}\times \mathrm{SO}(2) \to \mathrm{SO}(2),\,\,
 (\beta,\mathrm{r})\longmapsto {\widehat{\mathfrak{a}}(\beta\,
  \mathrm{r}^{-1})}^{-1}\, \mathrm{r}\, \widehat{\mathfrak{a}}(\beta)\, \mathrm{r}^{-1}.
   \]
To conclude the proof it suffices to show that $\widehat{\eta}$ is homotopic to a constant map.
Since $\Omega_{\mathrm{IV}}$ is star-shaped with respect to the origin $\mathrm{O}_{\mathrm{IV}}$,
we can define the smooth map
\[
 \upsilon : \Omega_{\mathrm{IV}}\times \mathrm{SO}(2)\times \R \to \mathrm{SO}(2), \,
  (\beta,\mathrm{r},t) \longmapsto  \widehat{\eta}(t\beta,\mathrm{r}).
  \]
 By construction, $\upsilon(\beta,\mathrm{r}, 1)=\widehat{\eta}(\beta,\mathrm{r})$
and $\upsilon(\beta,\mathrm{r},0)=I_{2}$. This shows that $\upsilon$ is a homotopy between the
constant map $I_{2}$ and $\widehat{\eta}$, which proves the claim.
\end{proof}

 Let
\begin{itemize}
\item $\mathfrak{m}:\Omega_{\mathrm{IV}}\times \Omega_{\mathrm{IV}}\to \Omega_{\mathrm{IV}}$,
\item $\mathfrak{r}:\Omega_{\mathrm{IV}}\times \Omega_{\mathrm{IV}}\to \mathrm{SO}(2)$, and
\item $\mathfrak{R}:\Omega_{\mathrm{IV}}\times \Omega_{\mathrm{IV}}\to \mathrm{SO}(n+1)$
\end{itemize}
be the smooth maps defined by requiring that
\[
 \mathrm{P}(\beta)  \mathrm{P}(\beta')
 =\mathrm{P}(\mathfrak{m}(\beta,\beta'))
  \begin{pmatrix}
    \mathfrak{r}(\beta,\beta') & 0 \\
    0 & \mathfrak{R}(\beta,\beta') \\
  \end{pmatrix},
    \quad \forall\, \beta,\,\beta'\in \Omega_{\mathrm{IV}}.
    \]
Since $\Omega_{\mathrm{IV}}$ is simply connected, there exists a unique map
$\Theta : \Omega_{\mathrm{IV}}\times \Omega_{\mathrm{IV}}\to \R$,
 such that
$\mathfrak{r}(\beta,\beta') =\rho[\Theta(\beta,\beta')]$ and
$\Theta(\mathrm{O}_{\mathrm{IV}},\mathrm{O}_{\mathrm{IV}})=0$. Let
\begin{itemize}
\item $\psi : \O(2,n+1)\to \mathrm{SO}(2)$,
\item $\Psi : \O(2,n+1)\to \mathrm{SO}(n+1)$, and
\item $\zeta :\O(2,n+1)\times \O(2,n+1)\to \R$
\end{itemize}
be the maps defined by requiring that
\[
 {\mathbf{X}}=
 \mathrm{P}\left(\pi^-_2(\mathrm{\mathbf{X}})\right)\,
                   \begin{pmatrix}
                     \psi(\mathrm{\mathbf{X}}) & 0 \\
                      0 & \Psi(\mathrm{\mathbf{X}}) \\
                     \end{pmatrix}
                \]
and
\[
 \zeta(\mathrm{\mathbf{X}},\mathrm{\mathbf{X}}') =
\Theta\left(\pi_2^-(\mathrm{\mathbf{X}}),\Psi(\mathrm{\mathbf{X}})\, \pi_2^-(\mathrm{\mathbf{X}}')\, \psi(\mathrm{\mathbf{X}})^{-1}\right)+\eta\left(\pi_2^-(\mathrm{\mathbf{X}}'),\psi(\mathrm{\mathbf{X}})\right),
    \]
for each $\mathrm{\mathbf{X}},\mathrm{\mathbf{X}}' \in \O(2,n+1)$.

\vskip0.1cm
We are now in a position to state the first main result of the paper.

\begin{thmx}\label{ThmA}
The subset of $\O(2,n+1)\times \R$ given by
\[
\OO(2,n+1)= \left\{(\mathrm{\mathbf{X}},\tau)\in \O(2,n+1)\times \R \mid \psi(\mathrm{\mathbf{X}})=\rho(\tau)\right\}
   \]
is a connected embedded submanifold diffeomorphic to
$\Omega_{\mathrm{IV}}\times \R\times \mathrm{SO}(n+1)$.
The multiplication
\[
  (\mathrm{\mathbf{X}},\tau)\star (\mathrm{\mathbf{X}}',\tau')=(\mathrm{\mathbf{X}}\, \mathrm{\mathbf{X}}',\tau+\tau'+
\zeta(\mathrm{\mathbf{X}},\mathrm{\mathbf{X}}'))
   \]
gives $\OO(2,n+1)$ the structure of a Lie group
with neutral element $({I}_{n+3},0)$ and inverse
$(\mathrm{\mathbf{X}},\tau)^{-1}=(\mathrm{\mathbf{X}}^{-1},-\tau-\zeta(\mathrm{\mathbf{X}},\mathrm{\mathbf{X}}^{-1}))$.
Moreover, the map
\[
   \sigma :  \OO(2,n+1)\ni (\mathrm{\mathbf{X}},\tau)\longmapsto \mathrm{\mathbf{X}}\in \O(2,n+1)
   \]
is a covering homomorphism of Lie groups.
If $\widehat{\mathrm{Z}}(2,n+1)$ denotes the center of $\OO(2,n+1)$, then

\begin{enumerate}
\item 
$\widehat{\mathrm{Z}}(2,n+1)=\left\{(I,2\pi k) \mid k\in \Z \right\}$ {$\cong \Z$}, if $n$ is even;

\item 
$\widehat{\mathrm{Z}}(2,n+1)=\left\{((-1)^kI,\pi k) \mid k\in \Z \right\}$ {$\cong \Z_2 \times \Z$}, if $n$ is odd.
\end{enumerate}
\end{thmx}

The proof of Theorem A is organized in four lemmas.

\begin{lemmaA}\label{lemma:A1}
The subset $\OO(2,n+1)$ is an embedded submanifold diffeomorphic to
$\Omega_{\mathrm{IV}}\times \R\times \mathrm{SO}(n+1)$ and $\sigma$ is a covering map.
\end{lemmaA}

\begin{proof}
Let $(\mathrm{\mathbf{X}}_*,\tau_*)$ be an element of $\OO(2,n+1)$ and $U\subset \O(2,n+2)$ be a contractible
open neighborhood of $\mathrm{\mathbf{X}}_*$. Then there exists a unique differentiable function
$\upsilon : U \to \R$
such that $\psi(\mathrm{\mathbf{X}})=\rho(\upsilon(\mathrm{\mathbf{X}}))$,
for every $\mathrm{\mathbf{X}}\in U$, satisfying $\upsilon(\mathrm{\mathbf{X}}_*) =\tau_*$.
Choose $\epsilon\in (0,\pi)$ and let
\[
  U'=\left\{\mathrm{\mathbf{X}}\in U \mid \upsilon(\mathrm{\mathbf{X}})\in (\tau_*-\epsilon,\tau_*+\epsilon)\right\}.
  \]
Then, $\widetilde{U}'=U'\times (\tau_*-\epsilon,\tau_*+\epsilon)\subset\O(2,n+1)\times \R$ is an open neighborhood
of $(\mathrm{\mathbf{X}}_*,\tau_*)$ such that
\[
  \widetilde{U}'\cap \OO(2,n+1)=\{(\mathrm{\mathbf{X}},\tau)\in \widetilde{U}' \mid \tau =\upsilon(\mathrm{\mathbf{X}})\}.
   \]
Hence, $\OO(2,n+1)\cap \widetilde{U}'$ is the graph of the function $\upsilon:U'\to \R$.
This implies that $\OO(2,n+1)$ is a submanifold of $\O(2,n+1)\times \R$. Clearly, the map
\[
   \Omega_{\mathrm{IV}}\times \R\times \mathrm{SO}(n+1) \ni (\beta,\tau,\mathrm{R})\longmapsto
   (\mathrm{P}(\beta)\, \mathrm{\mathbf{S}}(\rho(\tau),\mathrm{R}),\tau)\in \OO(2,n+1)
   \]
is bijective and of maximal rank. Thus, it is a diffeomorphism of
$\Omega_{\mathrm{IV}}\times \R\times \mathrm{SO}(n+1)$ onto $\OO(2,n+1)$.

By construction, $\sigma$ is a smooth surjective submersion. So, to conclude the proof
it suffices to prove that each $\mathrm{\mathbf{X}}_*\in \O(2,n+1)$ has an open neighborhood
which is evenly covered by $\sigma$.
Choose $\tau_*\in \R$ such that $\rho(\tau_*)=\psi(\mathrm{\mathbf{X}}_*)$ and let $U_*\subset \O(2,n+1)$
be the open neighborhood
\[
  U_*=\left\{\mathrm{\mathbf{X}}\in \O(2,n+1) \mid \psi(\mathrm{\mathbf{X}})\neq \rho(\tau_*+\pi)\right\}.
    \]
We now prove that $ U_*$ is evenly covered. For each $k\in \Z$ we consider the open neighborhood of $\OO(2,n+1)$ defined by
\[
  \widehat{U}_k=\left\{(\mathrm{\mathbf{X}},\tau)\in \OO(2,n+1) \mid
   \mathbf{X}\in U_*, \, \tau \in (\tau_*+2\pi k-\pi,\tau_*+2\pi k+\pi)\right\}.
   \]
Obviously, $\widehat{U}_k\cap \widehat{U}_k=\emptyset$, for every $h,k\in \Z$, $h\neq k$.
By construction, $\bigcup_{k\in \Z}\widehat{U}_k\subset \sigma^{-1}(U_*)$.
Let $(\mathrm{\mathbf{X}}',\tau')$ be an element of $\sigma^{-1}(U_*)$.
Then $\mathrm{\mathbf{X}}'\in U_*$ and $\rho(\tau')=\psi(\mathrm{\mathbf{X}}')\neq \rho(\tau_*+\pi)$.
Hence, $\tau'\neq \tau_* + \pi$ $\mod 2\pi \Z$. Therefore, there exists a unique $k\in \Z$ such
that $\tau'\in (\tau_*-\pi+2\pi k, \tau_*+\pi +2\pi k)$. This implies that $(\mathrm{\mathbf{X}}',\tau')\in \widehat{U}_k$. Consequently, $\bigcup_{k\in \Z}\widehat{U}_k$ and $\sigma^{-1}(U_*)$ coincide.
Choose $k\in \Z$ and let $\sigma_k:\widehat{U}_k\to U_*$ be the restriction of the map $\sigma$ to $\widehat{U}_k$.
Then, $\sigma_k$ is a differentiable bijection of maximal rank and hence is a homeomorphism.
This concludes the proof.
\end{proof}

\begin{lemmaA}\label{lemma:A2}
The maps $\psi$, $\zeta$, and $\rho$ are related by
\[
 \psi(\mathrm{\mathbf{X}}\, \mathrm{\mathbf{X}}')=\psi(\mathrm{\mathbf{X}})\, \psi(\mathrm{\mathbf{X}}')\,\rho\left(\zeta(\mathrm{\mathbf{X}},\mathrm{\mathbf{X}}')\right),
 \quad \forall \,\, \mathrm{\mathbf{X}},\mathrm{\mathbf{X}}'\in \O(2,n+1).
   \]
\end{lemmaA}

\begin{proof}
First we show that, for each $\beta \in \Omega_{\mathrm{IV}}$ and $\mathrm{\mathbf{S}}(\mathrm{r},\mathrm{R})\in \mathrm{SO}(2)\times \mathrm{SO}(n+1)$,
\begin{equation}\label{ee1}
\psi\left(\mathrm{\mathbf{S}}(\mathrm{r},\mathrm{R})\, \mathrm{P}(\beta)\right)
=\rho\left(\eta(\beta,\mathrm{r}) \right)\, \mathrm{r}.
     \end{equation}
By Lemma \ref{lemmaa},
$\widehat{\mathfrak{a}}(\mathrm{R}\, \beta)=\widehat{\mathfrak{a}}(\beta)$, for each
$\beta \in \Omega_{\mathrm{IV}}$ and $\mathrm{R}\in \mathrm{SO}(n+1)$.
Arguing as in the proof of Lemma \ref{lemmaa}, from
\begin{equation}\label{ee2}
\mathrm{\mathbf{S}}(\mathrm{r},\mathrm{R})\, \mathrm{P}(\beta)=
   \begin{pmatrix}
    \mathrm{r}\, \widehat{\mathfrak{a}}(\beta) & \mathrm{r}\, \widehat{\mathfrak{b}}(\beta) \\
    \mathrm{R}\, \widehat{\mathfrak{c}}(\beta) & \mathrm{R}\, \widehat{\mathfrak{d}}(\beta) \\
  \end{pmatrix}
\end{equation}
it follows that $\mathrm{\mathbf{S}}(\mathrm{r},\mathrm{R})\, \mathrm{P}(\beta)\in (\pi^-_2)^{-1}(\mathrm{R}\, \beta\, \mathrm{r}^{-1})$.
Therefore, we can write
\begin{equation}\label{ee3}
\mathrm{\mathbf{S}}(\mathrm{r},\mathrm{R})\,\mathrm{P}(\beta)=\mathrm{P}(\mathrm{R}\, \beta\,\mathrm{r}^{-1})\,
  \begin{pmatrix}
    \psi(\mathrm{\mathbf{S}}(\mathrm{r},\mathrm{R})\, \mathrm{P}(\beta)) & 0 \\
    0 & \Psi(\mathrm{\mathbf{S}}(\mathrm{r},\mathrm{R})\, \mathrm{P}(\beta)) \\
  \end{pmatrix}.
\end{equation}
By Lemma \ref{lemmaa}, comparing (\ref{ee2}) and (\ref{ee3}),
we obtain
\[
\begin{split}\psi(\mathrm{\mathbf{S}}(\mathrm{r},\mathrm{R})\, \mathrm{P}(\beta))&=
{\widehat{\mathfrak{a}}(\mathrm{R}\, \beta\,\mathrm{r}^{-1})}^{-1}\, \mathrm{r}\, \widehat{\mathfrak{a}}(\beta))
\\& = {\widehat{\mathfrak{a}}(\beta\, \mathrm{r}^{-1})}^{-1}\, \mathrm{r}\, \widehat{\mathfrak{a}}(\beta)\\ &=
\rho\left(\eta(\beta,\mathrm{r})\right)\, \mathrm{r}.
\end{split}
\]
Let $\mathrm{\mathbf{X}}$ and $\mathrm{\mathbf{X}}'$ be two elements of $\O(2,n+1)$ and write
\[
\mathrm{\mathbf{X}}=
\mathrm{P}(\beta)\,
                  \begin{pmatrix}
                      \psi(\mathrm{\mathbf{X}}) & 0 \\
                      0 & \Psi(\mathrm{\mathbf{X}}) \\
                                      \end{pmatrix},
                                 \quad
\mathrm{\mathbf{X}}'
=\mathrm{P}(\beta')\,
                  \begin{pmatrix}
                                        \psi(\mathrm{\mathbf{X}}') & 0 \\
                                        0 & \Psi(\mathrm{\mathbf{X}}') \\
                                     \end{pmatrix},
      \]
where $\beta = \pi^-_2(\mathrm{\mathbf{X}})$ and $\beta'= \pi^-_2(\mathrm{\mathbf{X}}')$.
By \eqref{ee1} and \eqref{ee3}, we obtain
\[
 \begin{split}
\mathrm{\mathbf{X}}\, \mathrm{\mathbf{X}}' &=
\mathrm{P}(\beta)
                  \left(\begin{smallmatrix}
                          \psi(\mathrm{\mathbf{X}}) & 0 \\
                        0 & \Psi(\mathrm{\mathbf{X}}) \\
                           \end{smallmatrix}\right)
                                \,
\mathrm{P}(\beta')
               \left(\begin{smallmatrix}
                            \psi(\mathrm{\mathbf{X}}') & 0 \\
                            0 & \Psi(\mathrm{\mathbf{X}}') \\
                           \end{smallmatrix}\right)
                                   \\
&= \mathrm{P}(\beta)\, \mathrm{P}(\Psi(\mathrm{\mathbf{X}})\, \beta'\, \psi(\mathrm{\mathbf{X}})^{-1})
\left(\begin{smallmatrix}
          \rho(\eta(\beta', \psi(\mathrm{\mathbf{X}}))\, \psi(\mathrm{\mathbf{X}})\, \psi(\mathrm{\mathbf{X}}') & 0\\
          0 & * \\
        \end{smallmatrix}\right)
 \\
 &=\mathrm{P}\left(\mathfrak{m}(\beta,\Psi(\mathrm{\mathbf{X}})\, \beta'\,\psi(\mathrm{\mathbf{X}})^{-1}) \right)
\left(\begin{smallmatrix}
    \rho(\Theta(\beta,\Psi(\mathrm{\mathbf{X}})\,\beta'\,
\psi(\mathrm{\mathbf{X}})^{-1})+\eta(\beta', \psi(\mathrm{\mathbf{X}}))\,
\psi(\mathrm{\mathbf{X}})\, \psi(\mathrm{\mathbf{X}}') & 0 \\
    0 & * \\
  \end{smallmatrix}\right) \\
& =\mathrm{P}\left(\mathfrak{m}(\beta,\Psi(\mathrm{\mathbf{X}})\, \beta'\, \psi(\mathrm{\mathbf{X}})^{-1}) \right)
\left(\begin{smallmatrix}
    \rho(\zeta(\mathrm{\mathbf{X}},\mathrm{\mathbf{X}}'))\, \psi(\mathrm{\mathbf{X}})\, \psi(\mathrm{\mathbf{X}}') & 0 \\
    0 & * \\
  \end{smallmatrix}\right).
\end{split}
\]
This implies that
\[
\psi(\mathrm{\mathbf{X}}\, \mathrm{\mathbf{X}}')=\psi(\mathrm{\mathbf{X}})\, \psi(\mathrm{\mathbf{X}}')\,
\rho(\zeta(\mathrm{\mathbf{X}},\mathrm{\mathbf{X}}')),
\]
as claimed. \end{proof}

\begin{lemmaA}\label{lemma:A3}
%
Let $\OO(2,n+1)$ be defined as above. The multiplication
\[
(\mathrm{\mathbf{X}},\tau)\star (\mathrm{\mathbf{X}}',\tau')=(\mathrm{\mathbf{X}}\, \mathrm{\mathbf{X}}',\tau+\tau'+
\zeta(\mathrm{\mathbf{X}},\mathrm{\mathbf{X}}'))
   \]
defines a Lie group structure on $\OO(2,n+1)$ with neutral element $\mathbf{1}=({I}_{n+3},0)$ and inverse
$(\mathrm{\mathbf{X}},\tau)^{-1}=(\mathrm{\mathbf{X}}^{-1},-\tau-\zeta(\mathrm{\mathbf{X}},\mathrm{\mathbf{X}}^{-1}))$.
\end{lemmaA}


\begin{proof}
First, we show that
$\mathbf{1}=({I}_{n+3},0)$
is a neutral element for the multiplication.
By definition, $\mathbf{1}\star (\mathrm{\mathbf{X}},\tau)\in \sigma^{-1}(\mathrm{\mathbf{X}})$, for each $(\mathrm{\mathbf{X}},\tau)\in \OO(2,n+1)$.
Then we can write $\mathbf{1}\star (\mathrm{\mathbf{X}},\tau)=(\mathrm{\mathbf{X}},\tau
+ 2\pi k(\mathrm{\mathbf{X}},\tau))$, where $k:\OO(2,n+1)\to \R$ is a smooth map with integral values.
Since $\OO(2,n+1)$ is connected, $k$ is constant. By construction, $k(\mathbf{1})=0$.
This implies that  $\mathbf{1}$ is a left neutral element.
The same reasoning shows that $\mathbf{1}$ is also a right neutral element.
Next, we prove that each $(\mathrm{\mathbf{X}},\tau)\in \OO(2,n+1)$ has a right inverse. Choose and fix $(\mathrm{\mathbf{X}}^{-1},\tau')\in \sigma^{-1}(\mathrm{\mathbf{X}}^{-1})$. Then
\[
(\mathrm{\mathbf{X}},\tau)\star (\mathrm{\mathbf{X}}^{-1},\tau') = ({I}_{n+3},\tau+\tau'+\zeta(\mathrm{\mathbf{X}},
\mathrm{\mathbf{X}}^{-1}))\in \sigma^{-1}({I}_{n+3}).
  \]
This implies that $\tau+\tau'+\zeta(\mathrm{\mathbf{X}},\mathrm{\mathbf{X}}^{-1}) = 2\pi m$, for some $m\in \Z$,
from which it follows that $(\mathrm{\mathbf{X}},\tau)$ $\star$ $(\mathrm{\mathbf{X}}^{-1},\tau'-2\pi m)=\mathbf{1}$.
We now prove that the right inverse is also a left inverse. Let $(\mathrm{\mathbf{X}},\tau)\in \OO(2,n+1)$ and let
$(\mathrm{\mathbf{X}}^{-1},\tau')$ be a right inverse of $(\mathrm{\mathbf{X}},\tau)$. We then have $\tau'=-\tau-\zeta(\mathrm{\mathbf{X}},\mathrm{\mathbf{X}}^{-1})$, and hence
\[
(\mathrm{\mathbf{X}}^{-1},\tau')\star (\mathrm{\mathbf{X}},\tau) =
 ({I}_{n+3},\zeta(\mathrm{\mathbf{X}}^{-1},\mathrm{\mathbf{X}})-
 \zeta(\mathrm{\mathbf{X}},\mathrm{\mathbf{X}}^{-1})).
\]
This implies that the image of the smooth map
\[
\widetilde{\zeta}: \O(2,n+1)\ni \mathrm{\mathbf{X}}\longmapsto \zeta(\mathrm{\mathbf{X}}^{-1},\mathrm{\mathbf{X}})-
 \zeta(\mathrm{\mathbf{X}},\mathrm{\mathbf{X}}^{-1})\in \R
  \]
 belongs to $2\pi \Z$, from which it follows that $\widetilde{\zeta}$ is constant.
Since $\widetilde{\zeta}({I}_{n+3})=0$, we conclude that $\widetilde{\zeta}$ vanishes identically.
Thus $(\mathrm{\mathbf{X}}^{-1},\tau')$ is a left inverse and the map
\[
 \OO(2,n+1)\ni (\mathrm{\mathbf{X}},\tau) \longmapsto (\mathrm{\mathbf{X}},\tau)^{-1}
 = (\mathrm{\mathbf{X}}^{-1},-\tau-\zeta(\mathrm{\mathbf{X}},\mathrm{\mathbf{X}}^{-1}))
 \in \OO(2,n+1)
\]
 is differentiable. We now prove that $\star$ is associative. Let $(\mathrm{\mathbf{X}}',\tau')$
 and $(\mathrm{\mathbf{X}}'',\tau'')$ be two elements of $\OO(2,n+1)$. Let
 $\Xi :\OO(2,n+1)\to \OO(2,n+1)$ be defined by
\[
\Xi (\mathrm{\mathbf{X}},\tau)= \left(((\mathrm{\mathbf{X}}',\tau')\star (\mathrm{\mathbf{X}}'',\tau'') )\star (\mathrm{\mathbf{X}},\tau)\right)\star
 \left((\mathrm{\mathbf{X}}',\tau')\star((\mathrm{\mathbf{X}}'',\tau'')\star (\mathrm{\mathbf{X}},\tau)) \right)^{-1}.
   \]
Since the multiplication of $\O(2,n+1)$ is associative,
$\Xi (\mathrm{\mathbf{X}},\tau)$ belongs to $\sigma^{-1}({I}_{n+3})$. We can
thus write $\Xi (\mathrm{\mathbf{X}},\tau)=({I}_{n+3},h(\mathrm{\mathbf{X}},\tau))$, where
$h : \OO(2,n+1)\to \R$ is a smooth function taking values in $2\pi \Z$.
This implies that $h$ is constant. On the other hand, $h(\mathbf{1})=0$, which
implies $\Xi=\mathbf{1}$, and hence the associativity of the product.
This concludes the proof.
\end{proof}


\begin{lemmaA}\label{lemma:A4}
Let $\widehat{\mathrm{Z}}(2,n+1)$ denote the center of $\OO(2,n+1)$. Then,
\begin{itemize}
\item 
$\widehat{\mathrm{Z}}(2,n+1)=\left\{(I,2\pi k)\mid k\in \Z \right\}$, if $n$ is even;

\item  
$\widehat{\mathrm{Z}}(2,n+1)=\left\{((-1)^kI,\pi k)\mid k\in \Z\right\}$, if $n$ is odd.
\end{itemize}
\end{lemmaA}

%

\begin{proof}
If $n$ is even, the center $\mathrm{Z}(2,n+1)$ of $\O(2,n+1)$ is trivial, while if $n$
is odd, the center is $\{\pm I_{n+3}\}$. Thus, if $n$ is even,
$\sigma^{-1}(\mathrm{Z}(2,n+1))= \left\{(I,2\pi k) \mid k\in \Z\right\}$,
 and if $n$ is odd, $\sigma^{-1}(\mathrm{Z}(2,n+1))= \left\{((-1)^kI,\pi k) \mid k\in \Z\right\}$.
By construction, $\widehat{\mathrm{Z}}(2,n+1)\subset \sigma^{-1}(\mathrm{Z}(2,n+1))$.
Suppose first that $n$ is even.
Take $(I,2\pi k)$, $k\in \Z$, and consider the smooth map
\[
   f_k : \OO(2,n+1)\ni  (\mathrm{\mathbf{X}},\tau) \longmapsto (\mathrm{\mathbf{X}},\tau)^{-1}\star (I,2\pi k)
 \star (\mathrm{\mathbf{X}},\tau)\in \OO(2,n+1).
   \]
Since $\sigma\circ f_k=I$, the image $\mathrm{Im}(f_k)\subset \left\{(I,2\pi m)\mid m\in \Z\right\}$,
from which it follows that
$f_k$ is constant. On the other hand, $f_k(I,0)=(I,2\pi k)$, and hence
$f_k(\mathrm{\mathbf{X}},\tau) = (I,2\pi k)$.
This implies that $(I,2\pi k)$ belongs to the center of $\OO(2,n+1)$,
which shows that $\widehat{\mathrm{Z}}(2,n+1)=\left\{(I,2\pi k)\mid k\in \Z \right\}$,
as claimed.
%
Suppose now that $n$ is odd. Arguing as above, we can show that
$\left\{(I,2\pi m)\mid m\in \Z\right\}\subset \widehat{\mathrm{Z}}(2,n+1)$. Next, take $(-I,\pi(1+2k))$,
$k\in \Z$, $k\neq 0$. Then the smooth map
\[
  \widetilde{f}_k : \OO(2,n+1)\ni (\mathrm{X},\tau)\longmapsto (\mathrm{\mathbf{X}},\tau)^{-1}\star (-I,\pi (1+2k))
 \star (\mathrm{\mathbf{X}},\tau)\in \OO(2,n+1)
    \]
 covers the constant map $-I$. Thus
 $\mathrm{Im}(\widetilde{f}_k)\subset \left\{(-I,\pi (1+2m))\mid m\in \Z\right\}$ and hence
$\widetilde{f}_k$ is constant. On the other hand, $\widetilde{f}_k(I,0)= (-I,\pi (1+2k))$, which implies $\widetilde{f}_k(\mathrm{\mathbf{X}},\tau)=  (-I,\pi (1+2k))$, for every $(\mathrm{\mathbf{X}},\tau)\in \OO(2,n+1)$.
From this it follows that $\left\{(-I,\pi (1+2k))\mid k\in \Z\right\}\subset \widehat{\mathrm{Z}}(2,n+1)$.
\end{proof}


The four lemmas combine to give the proof of Theorem \ref{ThmA}.


\begin{remark}\label{RM}
Let $\mathrm{s}:\mathrm{Spin}(2,n+1)\to \O(2,n+1)$ be the 2:1 spin covering group of $\O(2,n+1)$ (cf. \cite{LM}). Then, the universal covering group of $\O(2,n+1)$ is the embedded submanifold of $\mathrm{Spin}(2,n+1)\times \R$ defined by
\[
\widehat{\mathrm{S}}(2,n+1)=\left\{(\mathbf{X},\tau)\in \mathrm{Spin}(2,n+1)\times \R \mid
\psi(s(\mathbf{X}))= \rho(\tau)\right\}
 \]
 with the multiplication
 $(\mathbf{X},\tau)\ast (\mathbf{X}',\tau')=(\mathbf{X}\mathbf{X}',\tau+\tau'+\zeta(s(\mathbf{X}),s(\mathbf{X}')))$.
 Note that $\widehat{\mathrm{S}}(2,3)$ is the universal covering group of $\mathrm{Sp}(4,\R)$ while
 $\widehat{\mathrm{S}}(2,4)$ is the universal covering group of $\mathrm{SU}(2,2)$.
In particular, this implies that $\OO(2,n+1)$ cannot have a finite dimensional matrix representation
(see, for instance, \cite{Ra}).
\end{remark}

\section{Lorentz manifolds with conformal group of maximal dimension}\label{s:4}

In this section, we prove that the central extension $\OO(2,n+1)$ of $\O(2,n+1)$ is isomorphic
to the restricted conformal group of
the Einstein static universe and show that the integral compact forms of the first and second kind
admit a restricted conformal
transformation group of maximal dimension.
%
%
Conversely, we prove that a conformal Lorentz manifold of dimension $n+1 \geq 3$ with restricted conformal group
of maximal dimension is conformally equivalent to either the Einstein static universe, or to one of its
integral compact forms.

\subsection{The restricted conformal group of $\E$ and its integral compact forms}

\begin{defn}\label{d:ThTh'}
Let define
\begin{itemize}
\item $\mathcal{T}_h=\left\{(I_{n+3},2\pi hk) \mid k\in \Z\right\}$, for any integer $h\geq 1$,
\item  $\mathcal{T}'_h=\left\{\left((-1)^k I_{n+3},\pi (2h+1)k\right) \mid k\in \Z\right\}$,
for any integer $h\geq 0$ and $n$ odd.
\end{itemize}
Both $\mathcal{T}_h$ and $\mathcal{T}'_h$ are discrete subgroups of the center $\widehat{\mathrm{Z}}(2,n+1)$
of $\OO(2,n+1)$. Let
\[
  \widehat{\mathrm{O}}^{\uparrow,\mathrm{I}}_{+,h}(2,n+1) :=\OO(2,n+1)/\mathcal{T}_h,
  \quad \widehat{\mathrm{O}}^{\uparrow,\mathrm{II}}_{+,h}(2,n+1) :=\OO(2,n+1)/\mathcal{T}'_h
    \]
be the corresponding quotient Lie groups.
\end{defn}


We now state the second main result of the paper.

\begin{thmx}\label{ThmB}

$\mathrm(1)$ The restricted conformal group of the Einstein static universe $\E$ is isomorphic to $\OO(2,n+1)$.

$\mathrm(2)$ The restricted conformal group of the integral compact form of the first kind with index $h$ is isomorphic to
$\widehat{\mathrm{O}}^{\uparrow,\mathrm{I}}_{+,h}(2,n+1)$.

$\mathrm(3)$ The restricted conformal group of the integral compact form of the second kind with index $h$ is isomorphic
to $\widehat{\mathrm{O}}^{\uparrow,\mathrm{II}}_{+,h}(2,n+1)$.
\end{thmx}

\begin{proof}
Consider the left action $\mathrm{L} : \O(2,n+1)\times \EC \to \EC$
of $\O(2,n+1)$ on the standard compact form of
the first kind $\EC$ defined by
\[
  \mathrm{L}_{\mathrm{\mathbf{X}}}(\mathrm{x},\mathrm{y})=
 \left(\frac{\mathfrak{a}(\mathrm{\mathbf{X}})\, \mathrm{x}+\mathfrak{b}(\mathrm{\mathbf{X}})\, \mathrm{y}}
 {\|\mathfrak{a}(\mathrm{\mathbf{X}})\, \mathrm{x}+\mathfrak{b}(\mathrm{\mathbf{X}})\, \mathrm{y} \|},\frac{\mathfrak{c}(\mathrm{\mathbf{X}})\, \mathrm{x}+\mathfrak{d}(\mathrm{\mathbf{X}})\, \mathrm{y}}
 {\|\mathfrak{c}(\mathrm{\mathbf{X}})\, \mathrm{x}+\mathfrak{d}(\mathrm{\mathbf{X}})\, \mathrm{y} \|}
\right).
 \]
 The deck transformations of the covering $\pi_{\mathrm{I},1}:\E\to \EC$ are the translations
 \[
   \mathrm{T}_k :  \E\ni (\tau,\mathrm{y}) \longmapsto (\tau+2\pi k,\mathrm{y})\in \E,\quad k\in \Z.
     \]
 Let $\xi : \Omega_{\mathrm{IV}}\times \R\times \E\to \SS^1$
 be the smooth map defined by
 \begin{equation}\label{xi}
 \xi\left(\beta,\vartheta,(\tau,\mathrm{y})\right)=
 \frac{\widehat{\mathfrak{a}}(\beta)\, \rho(\vartheta)\, \rho_1(\tau)+\widehat{\mathfrak{b}}(\beta)\, \mathrm{y}}
 {\|\widehat{\mathfrak{a}}(\beta)\, \rho(\vartheta)\, \rho_1(\tau)+\widehat{\mathfrak{b}}(\beta)\, \mathrm{y} \|}.
 \end{equation}
 Choose $O_{\E}={}^t\!\left(0, (1,0,\dots,0)\right)\in \E$ as an origin.
 Since $\Omega_{\mathrm{IV}}\times \R\times \E$ is simply connected, there exists a unique smooth map
 $\widehat{\xi} : \Omega_{\mathrm{IV}}\times \R\times \E\to \R$,
 such that
 \begin{equation}\label{ctfc}
\rho(\widehat{\xi}\left(\beta,\vartheta,(\tau,\mathrm{y})\right))=
   \xi\left(\beta,\vartheta,(\tau,\mathrm{y})\right),\\
    \quad \widehat{ \xi}\left(\mathrm{O}_{\mathrm{IV}},0,O_{\E}\right)=0,
   \end{equation}
for every $(\beta,\vartheta,(\tau,\mathrm{y}))\in \Omega_{\mathrm{IV}}\times \R\times \E$.
For $(\mathrm{\mathbf{X}},\vartheta)\in \OO(2,n+1)$, consider the smooth map $\widehat{\mathrm{L}}_{(\mathrm{\mathbf{X}},\vartheta)}:\E\to \E$ given by
\begin{equation}\label{ac1}
 \widehat{\mathrm{L}}_{(\mathrm{\mathbf{X}},\vartheta)}\big((\tau,\mathrm{y})\big)=
  \left(\widehat{\xi}\left(\pi^-_2(\mathrm{\mathbf{X}}),\vartheta,
    (\tau,\Psi(\mathrm{\mathbf{X}}) \, \mathrm{y})\right),
    \frac{\mathfrak{c}(\mathrm{\mathbf{X}})\, \rho_1(\tau)+\mathfrak{d}(\mathrm{\mathbf{X}})\, \mathrm{y}}
    {\|\mathfrak{c}(\mathrm{\mathbf{X}})\, \rho_1(\tau)+\mathfrak{d}(\mathrm{\mathbf{X}})\, \mathrm{y} \|}\right).
    \end{equation}
Let $\widehat{\mathrm{L}}:\OO(2,n+1)\times \E\to \E$ be defined by
\begin{equation}\label{ac2}
 \widehat{\mathrm{L}}\left((\mathrm{\mathbf{X}},\vartheta),(\tau,\mathrm{y})\right)=
  \widehat{\mathrm{L}}_{(\mathrm{\mathbf{X}},\vartheta)}\big((\tau,\mathrm{y})\big).
         \end{equation}
Then, $\widehat{\mathrm{L}}$ covers the action of $\O(2,n+1)$ on $\EC$, that is,
$\pi_{\mathrm{I},1}\circ \widehat{\mathrm{L}}= \mathrm{L}\circ (\sigma,\pi_{\mathrm{I},1})$.
This implies that $\widehat{\mathrm{L}}_{(\mathrm{\mathbf{X}},\vartheta)}$ is an orientation
and time-orientation preserving conformal transformation of $\E$.

\vskip0.2cm
To complete the proof of Theorem \ref{ThmB} we need the following.

\begin{lemmaB}\label{lemma:B1}
%
The map $\widehat{\mathrm{L}}$ defines an effective left action of $\OO(2,n+1)$ on the Einstein
static universe by restricted conformal transformations. Then, the restricted conformal group of $\E$ is
isomorphic to $\OO(2,n+1)$.
\end{lemmaB}


\begin{proof}[Proof of Lemma B\ref{lemma:B1}]
First, note that $\widehat{\mathrm{L}}_{(I,0)}=\mathrm{Id}_{\E}$. In fact, $\widehat{\mathrm{L}}_{(I,0)}$
is a deck transformation of the covering $\pi_{\mathrm{I},1}$ and
by definition and \eqref{ctfc}, $\widehat{\mathrm{L}}_{(I,0)}(O_{\E}) = O_{\E}$.
This implies that $\widehat{\mathrm{L}}_{(I,0)}=\mathrm{Id}_{\E}$. We now prove that
\[
 \widehat{\mathrm{L}}_{(\mathrm{\mathbf{X}}_*,\vartheta_*)^{-1}}\circ
  \widehat{\mathrm{L}}_{(\mathrm{\mathbf{X}}_*,\vartheta_*)} =
   \mathrm{Id}_{\E},\quad \forall \,\,(\mathrm{\mathbf{X}}_*,\vartheta_*)\in \OO(2,n+1).
    \]
The composition $\widehat{\mathrm{L}}_{(\mathrm{\mathbf{X}}_*,\vartheta_*)^{-1}}\circ \widehat{\mathrm{L}}_{(\mathrm{\mathbf{X}}_*,\vartheta_*)}$ is a deck transformation of
the covering $\pi_{\mathrm{I},1}$.
Consider a smooth path
$[0,1]\ni s\mapsto (\mathrm{\mathbf{X}}(s),\vartheta(s))\in \OO(2,n+1)$,
such that
$(\mathrm{\mathbf{X}}(0),\vartheta(0)) = (\mathrm{\mathbf{X}}_*,\vartheta_*)$
and $(\mathrm{\mathbf{X}}(1),\vartheta(1))=(I,0)$.
Put $\lambda_s = \widehat{\mathrm{L}}_{(\mathrm{\mathbf{X}}(s),\vartheta(s))^{-1}}\circ \widehat{\mathrm{L}}_{(\mathrm{\mathbf{X}}(s),\vartheta(s))}$ and consider the smooth map given by
\[
f :  [0,1]\times \E \ni (s,(\tau,\mathrm{y}))\longmapsto f(s,(\tau,\mathrm{y}))=\lambda_s\left((\tau,\mathrm{y})\right)\in \E.
  \]
By construction, $\lambda_s : \E \to \E$ is a deck transformation of the covering $\pi_{\mathrm{I},1}$, and hence $f(s,(\tau,\mathrm{y}))=(\tau+2\pi k,\mathrm{y})$,
{where $k: [0,1] \to \mathbb Z$, $s\mapsto k(s)$, is independent of $s$.
Since, by \eqref{ctfc}, $f(1,O_{\E})=O_{\E}$, i.e., $\lambda_1(O_{\E})=O_{\E}$,
we have that
$\lambda_s (O_{\E}) = O_{\E}$, and hence $\lambda_s =\mathrm{Id}_{\E}$
for every $s\in [0,1]$.}
In particular, $\lambda_0=\widehat{\mathrm{L}}_{(\mathrm{\mathbf{X}}_*,\vartheta_*)^{-1}}\circ \widehat{\mathrm{L}}_{(\mathrm{\mathbf{X}}_*,\vartheta_*)}=
\mathrm{Id}_{\E}$, which implies that $\widehat{\mathrm{L}}_{(\mathrm{\mathbf{X}}_*,\vartheta_*)}$ is a
conformal diffeomorphism of $\E$ with inverse $\widehat{\mathrm{L}}_{(\mathrm{\mathbf{X}}_*,\vartheta_*)^{-1}}$.
We now show that
\[
 \widehat{\mathrm{L}}_{(\mathrm{\mathbf{X}}_*,\vartheta_*)\star (\mathrm{\mathbf{X}}'_*,\vartheta'_*)}=
  \widehat{\mathrm{L}}_{(\mathrm{\mathbf{X}}_*,\vartheta_*)}\circ \widehat{\mathrm{L}}_{(\mathrm{\mathbf{X}}'_*,\vartheta'_*)},
    \]
for every $(\mathrm{\mathbf{X}}_*,\vartheta_*)$, $(\mathrm{\mathbf{X}}'_*,\vartheta'_*)$ $\in$ $\OO(2,n+1)$.
Let $\Phi_{((\mathrm{\mathbf{X}}_*,\vartheta_*),(\mathrm{\mathbf{X}}'_*,\vartheta'_*))}$ be the restricted
conformal automorphism of $\E$ defined by
\[
  \Phi_{((\mathrm{\mathbf{X}}_*,\vartheta_*),(\mathrm{\mathbf{X}}'_*,\vartheta'_*))}=
\left(\widehat{\mathrm{L}}_{(\mathrm{\mathbf{X}}_*,\vartheta_*)\star (\mathrm{\mathbf{X}}'_*,\vartheta'_*)}\right)^{-1}\circ
(\widehat{\mathrm{L}}_{(\mathrm{\mathbf{X}}_*,\vartheta_*)}\circ \widehat{\mathrm{L}}_{(\mathrm{\mathbf{X}}'_*,\vartheta'_*)}).
   \]
By definition, $\Phi_{((\mathrm{\mathbf{X}}_*,\vartheta_*),(\mathrm{\mathbf{X}}'_*,\vartheta'_*))}$
is a deck transformation of $\pi_{\mathrm{I},1}$.
Next, consider the smooth paths
$[0,1]\ni s \mapsto (\mathrm{\mathbf{X}}(s),\vartheta(s))\in \O(2,n+1)$ and
$[0,1]\ni s \mapsto (\mathrm{\mathbf{X}}'(s),\vartheta'(s))$ $\in$ $\O(2,n+1)$,
such that $(\mathrm{\mathbf{X}}(0),\vartheta(0))=(\mathrm{\mathbf{X}}_*,\vartheta_*)$,
$(\mathrm{\mathbf{X}}(1),\vartheta(1))=(I,0)$ and
$(\mathrm{\mathbf{X}}'(0),\vartheta'(0))=(\mathrm{\mathbf{X}}'_*,\vartheta'_*)$,
$(\mathrm{\mathbf{X}}'(1),\vartheta'(1))=(I,0)$, respectively.
Consider the differentiable map
\[
\widetilde{f}: [0,1]\times \E \ni (s,(\tau,\mathrm{y}))\longmapsto
\Phi_{((\mathrm{\mathbf{X}}(s),\vartheta(s)),(\mathrm{\mathbf{X}}'(s),\vartheta'(s)))}\in \E.
   \]
Then, for every $s\in [0,1]$, the map
$\widetilde{f}_s: \E \ni (\tau,\mathrm{y})\mapsto \widetilde{f}(s,(\tau,\mathrm{y}))\in \E$
is a deck transformation of $\pi_{\mathrm{I},1}$. Consequently, there exists $k\in \Z$, such that
$\widetilde{f}(s,(\tau,\mathrm{y})) = (\tau+2\pi k,\mathrm{y})$.
Since $\widetilde{f}_1=\mathrm{Id}_{\E}$, it follows that
$\widetilde{f}_s=\mathrm{Id}_{\E}$, for every $s\in [0,1]$. This implies that
\[
 \mathrm{Id}_{\E}=\widetilde{f}_0=\left(\widehat{\mathrm{L}}_{(\mathrm{\mathbf{X}}_*,\vartheta_*)
 \star (\mathrm{\mathbf{X}}'_*,\vartheta'_*)}\right)^{-1}\circ
  \left(\widehat{\mathrm{L}}_{(\mathrm{\mathbf{X}}_*,\vartheta_*)}\circ
   \widehat{\mathrm{L}}_{(\mathrm{\mathbf{X}}'_*,\vartheta'_*)}\right).
   \]
We now prove that the action $\widehat{\mathrm{L}}$ is effective. Suppose $\widehat{\mathrm{L}}_{(\mathrm{\mathbf{X}}_*,\vartheta_*)}=\mathrm{Id}_{\E}$.
Since the action of $\O(2,n+1)$ on $\EC$ is effective,
we have $\mathrm{\mathbf{X}}_*= I_{n+3}$ and $\vartheta_*=2\pi k$, where $k\in \Z$. Then,
$\widehat{\mathrm{L}}_{(I_{n+3},2\pi k)}(\tau,\mathrm{y})=(\tau +2\pi k, \mathrm{y})$.
Since $\widehat{\mathrm{L}}_{(\mathrm{\mathbf{X}}_*,\vartheta_*)}=\mathrm{Id}_{\E}$,
we must have $k=0$ and hence $(\mathrm{\mathbf{X}}_*,\vartheta_*)=(I_{n+3},0)$.
We have shown that $\OO(2,n+1)$ is a connected Lie
group acting effectively and transitively on $\E$ by restricted conformal transformations.
Since $\OO(2,n+1)$ has dimension $(n+3)(n+2)/2$,
it follows from Theorem \ref{THM0} that $\OO(2,n+1)$ is isomorphic to $\mathcal{C}^{\uparrow}_+(\E)$.
%
This concludes the proof of Lemma B\ref{lemma:B1}.
\end{proof}


We now resume the proof of Theorem \ref{ThmB}.
With reference to Definitions \ref{d:icf1stk}, \ref{d:icf2ndk} and \ref{d:ThTh'},
it is clear that $\mathcal{T}_h$ is the group of deck transformations of the covering
$\pi_{\mathrm{I},h} :\E\to \mathcal{E}^{1,n}_{\mathrm{I},h}$,
and similarly, if $n$ is odd, $\mathcal{T}'_h$ is the group of deck transformations of the covering $\pi_{\mathrm{II},h}:\E\to \mathcal{E}^{1,n}_{\mathrm{II},h}$. The left action of $\OO(2,n+1)$ on $\E$ descends to effective left actions
by restricted conformal transformations
$\mathrm{L}_{h}^{\mathrm{I}}: \widehat{\mathrm{O}}^{\uparrow,\mathrm{I}}_{+,h}(2,n+1)\times \mathcal{E}^{1,n}_{\mathrm{I},h}\to \mathcal{E}^{1,n}_{\mathrm{I},h}$ and
$\mathrm{L}_{h}^{\mathrm{II}}: \widehat{\mathrm{O}}^{\uparrow,\mathrm{II}}_{+,h}(2,n+1)\times \mathcal{E}^{1,n}_{\mathrm{II},h}\to \mathcal{E}^{1,n}_{\mathrm{II},h}$ on $ \mathcal{E}^{1,n}_{\mathrm{I},h}$ and $\mathcal{E}^{1,n}_{\mathrm{II},h}$, respectively.
%
%
Again by Theorem \ref{THM0}, it follows that $\widehat{\mathrm{O}}^{\uparrow,\mathrm{I}}_{+,h}(2,n+1)$
and $\widehat{\mathrm{O}}^{\uparrow,\mathrm{II}}_{+,h}(2,n+1)$ are
isomorphic
to the restricted conformal groups of $\mathcal{E}^{1,n}_{\mathrm{I},h}$ and $\mathcal{E}^{1,n}_{\mathrm{II},h}$,
respectively.
\end{proof}

It is important to observe that
any two integral compact forms with different indices belonging to the same series,
as well as any two integral compact forms belonging to different series, cannot be conformally
equivalent. More precisely, we have the following.


\begin{prop}\label{prop:inequiv}
$\mathrm(1)$ Any two integral compact forms
$\mathcal{E}^{1,n}_{\mathrm{I},k}$ and $\mathcal{E}^{1,n}_{\mathrm{I},h}$, $k\neq h$,
cannot be conformally diffeomorphic.
%
$\mathrm(2)$ Any two integral compact forms
$\mathcal{E}^{1,n}_{\mathrm{II},k}$ and $\mathcal{E}^{1,n}_{\mathrm{II},h}$, $k\neq h$,
cannot be conformally diffeomorphic.
%
$\mathrm(3)$ Any two integral compact forms
$\mathcal{E}^{1,n}_{\mathrm{II},k}$ and $\mathcal{E}^{1,n}_{\mathrm{I},h}$
cannot be conformally diffeomorphic.
%
%
\end{prop}

\begin{proof}
(1) Suppose
$f : \mathcal{E}^{1,n}_{\mathrm{I},k} \to \mathcal{E}^{1,n}_{\mathrm{I},h}$
is a conformal diffeomorphism.
Since $\mathcal{E}^{1,n}$ is simply connected, there exists
an orientation preserving conformal transformation  of maximal rank
$\tilde f : \mathcal{E}^{1,n} \to \mathcal{E}^{1,n}$ that covers $f$.
Let $\dot{\tilde{f}} : \OO(2,n+1) \to \OO(2,n+1)$ be the conformal prolongation of $\tilde f$.
Since $\dot{\tilde{f}}$ preserves the Maurer--Cartan form of $\OO(2,n+1)$, $\dot{\tilde{f}}$ coincides with
the left multiplication by an element $(\mathbf H,\tau)$ of $\OO(2,n+1)$. This implies that,
for each $(t, \mathrm x) \in \E$,
\[
 \tilde f\left((t, \mathrm x) \right) =
 \widehat{\mathrm{L}}_{({\mathbf{H}},\tau)}\big((t,\mathrm{x})\big).
   \]
Since $\tilde f$ covers $f$, $\tilde f$ takes the fibers of the covering
$\pi_{\mathrm I,k} : \E \to \mathcal{E}^{1,n}_{\mathrm{I},k}$  to the
fibers of the covering $\pi_{\mathrm I,h} : \E \to \mathcal{E}^{1,n}_{\mathrm{I},h}$.
In particular,
if we set $(\mathfrak{t}, \mathrm{y}) : = \widehat{\mathrm{L}}_{({\mathbf{H}},\tau)}\big((t,\mathrm{x})\big)$
and take into account that $\mathcal{T}_h$ (cf. Definition \ref{d:ThTh'}) is the deck transformation
group of the covering $\pi_{\mathrm{I},h}$, then
\begin{equation}\label{fiber-to-fiber}
 \widehat{\mathrm{L}}_{({\mathbf{H}},\tau)}\big((t + 2\pi k,\mathrm{x})\big)
 = (\mathfrak{t} +2\pi m h, \mathrm{y}), \quad m\in \mathbb Z.
   \end{equation}
Using the fact that $(I, 2\pi k)$ belongs to the center of $\OO(2,n+1)$, the left-hand-side of \eqref{fiber-to-fiber}
can be written as
\begin{equation}\label{fiber-to-fiber1}
\begin{split}
 \widehat{\mathrm{L}}_{({\mathbf{H}},\tau)}\big((t + 2\pi k,\mathrm{x})\big)
 &= \left(\widehat{\mathrm{L}}_{({\mathbf{H}},\tau)} \circ
     \widehat{\mathrm{L}}_{({I},2\pi k)} \right) \big((t,\mathrm{x})\big) \\
         &= \widehat{\mathrm{L}}_{({I},2\pi k)} \left(\widehat{\mathrm{L}}_{({\mathbf{H}},\tau)}
          \big((t,\mathrm{x})\big) \right)\\
          & = (\mathfrak{t} +2\pi k, \mathrm{y}).
       \end{split}
  \end{equation}
From \eqref{fiber-to-fiber} and \eqref{fiber-to-fiber1}, it follows that
 $k = mh$, $m \in \mathbb Z$. Since $h,k >0$, $m>0$. Repeating the argument
for the inverse map $f^{-1}$ yields $h=nk$, $n \in \mathbb Z$, $n>0$. In conclusion,
if $f$ is a conformal diffeorphism, then $k = h$.

(2) Suppose
$f : \mathcal{E}^{1,n}_{\mathrm{II},k} \to \mathcal{E}^{1,n}_{\mathrm{II},h}$
is a conformal diffeomorphism.
Let
%
%
$\tilde f : \mathcal{E}^{1,n} \to \mathcal{E}^{1,n}$ be the conformal transformation that covers $f$
and
let $\dot{\tilde{f}}$ be the conformal prolongation of $\tilde f$.
%
As above,
$\dot{\tilde{f}}$ coincides with
the left multiplication by an element $(\mathbf H,\tau)$ of $\OO(2,n+1)$
and, for each $(t, \mathrm x) \in \E$,
\[
 \tilde f\left((t, \mathrm x) \right) =
 \widehat{\mathrm{L}}_{({\mathbf{H}},\tau)}\big((t,\mathrm{x})\big).
   \]
As $\tilde f$ covers $f$, $\tilde f$ takes the fibers of the covering
$\pi_{\mathrm{II},k} : \E \to \mathcal{E}^{1,n}_{\mathrm{II},k}$  to the
fibers of the covering $\pi_{\mathrm{II},h} : \E \to \mathcal{E}^{1,n}_{\mathrm{II},h}$.
In particular,
taking into account that $\mathcal{T}'_h$ (cf. Definition \ref{d:ThTh'}) is the deck transformation
group of $\pi_{\mathrm{II},h}$, we have
\begin{equation}\label{fiber-to-fiber2}
 \widehat{\mathrm{L}}_{({\mathbf{H}},\tau)}\big((t + (2k+1)\pi , - \mathrm{x})\big)
 = \begin{cases}
  (\mathfrak{t} + \pi  (2h+1) 2q, \mathrm{y}), \quad q \in \mathbb Z,\\
   (\mathfrak{t} + \pi (2h+1) (2q +1), -\mathrm{y}), \quad q \in \mathbb Z,\\
 \end{cases}
   \end{equation}
where again $(\mathfrak{t}, \mathrm{y}) : = \widehat{\mathrm{L}}_{({\mathbf{H}},\tau)}\big((t,\mathrm{x})\big)$.
Using the fact that $(-I, (2k+1)\pi )$ belongs to the center of $\OO(2,n+1)$, the left-hand-side of \eqref{fiber-to-fiber2}
can be written as
\begin{equation}\label{fiber-to-fiber3}
\begin{split}
 \widehat{\mathrm{L}}_{({\mathbf{H}},\tau)}\big((t + (2k+1)\pi , - \mathrm{x})\big)
&= \left(\widehat{\mathrm{L}}_{({\mathbf{H}},\tau)} \circ
     \widehat{\mathrm{L}}_{(-{I},(2k+1)\pi )} \right) \big((t,\mathrm{x})\big) \\
         &= \widehat{\mathrm{L}}_{(-{I},(2k+1)\pi)} \left(\widehat{\mathrm{L}}_{({\mathbf{H}},\tau)}
          \big((t,\mathrm{x})\big) \right)\\
        &= (\mathfrak{t} + (2k+1)\pi , - \mathrm{y}).
       \end{split}
  \end{equation}
  From \eqref{fiber-to-fiber2} and \eqref{fiber-to-fiber3}, it follows that
 $2k +1 = (2q+1)(2h+1)$, $q \in \mathbb Z$. Since $h,k \geq 0$, $2q+1>0$. Repeating the argument
for the inverse map $f^{-1}$ yields $2h +1 = (2s+1)(2k+1)$, $s \in \mathbb Z$, $2s+1>0$.
Therefore, if $f$ is a conformal diffeomorphism, $k = h$.

(3) Seeking a contradiction, we suppose that
$f : \mathcal{E}^{1,n}_{\mathrm{II},k} \to \mathcal{E}^{1,n}_{\mathrm{I},h}$
is a conformal diffeomorphism.
Arguing as in the proof of points (1) and (2), we are led to
\[
 (\mathfrak{t} + (2k+1)\pi , - \mathrm{y}) = (\mathfrak{t} + 2\pi m h ,  \mathrm{y}), \quad m\in \mathbb Z,
  \]
which is the desired contradiction.
%
\end{proof}


\subsection{Lorentz manifolds with restricted conformal group of maximal dimension}

In this section, we prove the following result about Lorentz manifolds
with restricted conformal group of maximal dimension.


\begin{thmx}\label{ThmC}
Let $\mathbb{M}$ be a $(n+1)$-dimensional, $n\ge 2$, connected, oriented, time-oriented conformal Lorentz manifold
such that
\[
\mathrm{dim}(\mathcal{C}^{\uparrow}_+(\mathbb{M}))=\frac12(n+3)(n+2).
  \]
Then
\begin{enumerate}
\item If $\mathbb{M}$ is simply connected, then $\mathbb{M}$ is conformally equivalent to the Einstein universe $\E$.

\item If $\mathbb{M}$ is not simply connected and $n$ is even, then $\mathbb{M}$ is conformally equivalent to an integral compact form of the first kind of the Einstein  universe.
\item If $\mathbb{M}$ is not simply connected and $n$ is odd, then $\mathbb{M}$ is conformally equivalent to either an integral compact form of the first kind of the Einstein universe, or to an integral compact form of the second kind of the Einstein universe.
\end{enumerate}
\end{thmx}

\begin{proof}
If $\mathrm{dim}(\mathcal{C}^{\uparrow}_+(\mathbb{M}))=\frac12(n+3)(n+2)$, from Theorem \ref{THM0} it follows that:

\begin{itemize}

\item the Cartan conformal bundle $\mathrm Q(\mathbb{M})$ is a Lie group acting effectively on $\mathbb{M}$
by restricted conformal transformations;

\item the normal conformal connection of $\mathrm Q(\mathbb{M})$, denoted by $\phi'$, coincides with the Maurer--Cartan form;

\item Let $e'\in \mathrm Q(\mathbb{M})$ be the neutral element and put $p'_*=\pi_{\mathrm Q(\mathbb{M})}(e')$.
The isotropy subgroup $\mathrm{H}'\subset \mathrm Q(\mathbb{M})$ of the point $p_*'$ coincides with the fiber
$\pi_{\mathrm Q(\mathbb{M})}^{-1}(p'_*)$, and hence it is isomorphic to $\mathrm{H}^{\uparrow}_+(2,n+1)$;

\item $\mathrm{H}'$ is the maximal integral submanifold through $e'$ of the left-invariant completely integrable Pfaffian differential systems generated by the 1-forms ${\phi'}^j_0$, $j=1,\dots,n+1$.

\end{itemize}

\noindent Analogous conclusions hold for the Einstein static universe, namely:

\begin{itemize}

\item  $\mathrm Q(\E)\cong \OO(2,n+1)$ is a Lie group acting effectively on $\E$ by restricted conformal transformations;

\item the normal conformal connection of $\mathrm Q(\E)$, denoted by $\phi''$, coincides with the Maurer--Cartan form;

\item  Let $e''\in \mathrm Q(\E)$ be the neutral element and put $p''_*=\pi_{\mathrm Q(\E)}(e'')$. The stabilizer of the point $p_*''$ coincides with the fiber $\pi_{\mathrm Q(\E)}^{-1}(p''_*)$. It is a closed Lie
    subgroup $\mathrm{H}''\subset \mathrm Q(\E)$, isomorphic to $\mathrm{H}^{\uparrow}_+(2,n+1)$.


\item $\mathrm{H}''$ is the maximal integral sub-manifold through $e''$ of the left-invariant completely integrable Pfaffian differential systems generated by the 1-forms ${\phi''}^{j}_0$, $j=1,\dots,n+1$.

\end{itemize}

For point (1),
let $\widehat{\mathrm{S}}(2,n+1)$ be the universal covering group of $\mathrm Q(\E)$.
Since $\mathbb{M}$ is connected and the fibers of
$\mathrm{Q}(\mathbb{M})$ are connected, from the exact homotopy sequence of a principal bundle
\cite{St},
in this case $\pi_{\mathrm{Q}}:\mathrm{Q}(\mathbb{M})\to \mathbb{M}$, it follows that  $\mathrm{Q}(\mathbb{M})$ is connected.
Hence $\widehat{\mathrm{S}}(2,n+1)$ is also the universal covering group of $\mathrm{Q}(\mathbb{M})$. Let
\[
 \mathrm{pr}_1:\widehat{\mathrm{S}}(2,n+1)\to \mathrm{Q}(\mathbb{M}),\quad
  \mathrm{pr}_2: \widehat{\mathrm{S}}(2,n+1)\to \mathrm Q(\E)
    \]
be the corresponding covering homomorphisms. They can be chosen in a way that
$$
  \phi={\mathrm{pr}_1}^*(\phi')={\mathrm{pr}_2}^*(\phi''),
   $$
where $\phi$ is the Maurer--Cartan form of $\widehat{\mathrm{S}}(2,n+1)$. Let $e\in \widehat{\mathrm{S}}(2,n+1)$
be the multiplicative unit and $\widetilde{\mathrm{H}}$ be the maximal integral submanifold trough $e$ of the
completely integrable Pfaffian differential system generated by $\phi^1_0,\dots,\phi^{n+1}_0$.
The universal covering $\widehat{\mathrm{S}}(2,n+1)$ acts almost effectively on $\mathbb{M}$ and $\E$.
Let $\widetilde{\mathrm H}'$ and $\widetilde{\mathrm H}''$ be the stabilizers of these actions at $p'_*$ and $p''_*$,
respectively. Since $\widehat{\mathrm{S}}(2,n+1)$, $\mathbb{M}$, and $\E$ are connected and simply connected,
from the exact homotopy sequence of a principal bundle it follows that $\widetilde{\mathrm H}'$
and $\widetilde{\mathrm H}''$
are connected. By construction, they are integral manifolds of the Pfaffian differential
system $\phi^1_0=\cdots=\phi^{n+1}=0$. Hence $\widetilde{\mathrm H}'=\widetilde{\mathrm H}''=\widetilde{\mathrm{H}}$.
Therefore, $\mathbb{M}$ and $\E$ are both diffeomorphic to the homogeneous space $\widehat{\mathrm{S}}(2,n+1)/\widetilde{\mathrm{H}}$. Thus there exists a unique diffeomorphism
$\Phi:\mathbb{M}\to \E$, such that $\Phi\circ \widehat{\pi}_{\mathbb{M}}=\widehat{\pi}_{\E}$, where
\[
\widehat{\pi}_{\mathbb{M}}:\widehat{\mathrm{S}}(2,n+1)\to \mathbb{M},\quad
\widehat{\pi}_{\E}:\widehat{\mathrm{S}}(2,n+1)\to \E
   \]
are the two natural bundle maps.
Now we show that $\Phi$ is an orientation and time-orientation preserving conformal map.
We cover $\mathbb{M}$ with a family $\{U_{\alpha}\}_{\alpha\in \mathcal{C}'}$ of simply
connected open neighborhoods such that, for each $\alpha\in \mathcal{C}'$,
there exists a cross section $\dot{\mathcal{A}}_{\alpha}:U_{\alpha}\to \mathrm{Q}(\mathbb{M})$. For each $\alpha$,
we choose
a lift $\dot{\mathcal{A}}'_{\alpha}$ of $\dot{\mathcal{A}}_{\alpha}$ to $\widehat{\mathrm{S}}(2,n+1)$.
Then, the map
\[
 \dot{\mathcal{B}}_{\alpha}:=\mathrm{pr}_2\circ \dot{\mathcal{A}}'_{\alpha}\circ \Phi^{-1}:\Phi(U_{\alpha})\to \mathrm{Q}(\E)
  \]
is a cross section of $Q(\E)$. Put $\phi_{\alpha}=\dot{\mathcal{A}}_{\alpha}^*(\phi')$ and $\widetilde{\phi}_{\alpha}=\dot{\mathcal{B}}_{\alpha}^*(\phi'')$. Then,
$(\Phi^{-1})^*(\widetilde{\phi}_{\alpha})$ $=$ $\phi_{\alpha}$.
Hence, $({\phi_{\alpha}}^1_0,\dots,{\phi_{\alpha}}^{n+1}_0)$
and
$(\widetilde{\phi}_{\alpha,0}^{\hskip0.2cm 1},\dots,\widetilde{\phi}_{\alpha,0}^{\hskip0.2cm n+1})$
are two positive oriented and time-oriented conformal coframes, defined on the open neighborhoods
$U_{\alpha}\subset \mathbb{M}$ and $\Phi(U_{\alpha})\subset \E$, respectively, such
that $(\Phi^{-1})^*(\widetilde{\phi}_{\alpha,0}^{\hskip0.2cm i})= {\phi_{\alpha}}^i_0$, $i=1,\dots,n+1$.
This implies that the restriction of $\Phi$ to $U_{\alpha}$ is an orientation and time-orientation
preserving conformal diffeomorphism.

\vskip0.1cm

As for points (2) and (3), if $\mathbb{M}$ is not simply connected, let $\mathrm{p}_{\mathbb{M}}:\mathbb{M}^*\to \mathbb{M}$ be the universal covering space of $\mathbb{M}$, equipped with the oriented,
time-oriented Lorentz structure such that $\mathrm{p}_{\mathbb{M}}$ is conformal, orientation
and time-orientation preserving.
The group $\Gamma$ of deck transformations of the covering $\mathrm{p}_{\mathbb{M}}$ is contained in the
restricted conformal group $\mathcal{C}^{\uparrow}_+(\mathbb{M}^*)$ of $\mathbb{M}^*$.
Then, according to \cite[Theorem 9.1, page 63]{BRD},
there exists a covering group $\pi_{\mathcal{C}}:\mathrm{G}\to \mathcal{C}^{\uparrow}_+(\mathbb{M})$ and
an effective action
\[
  \widehat{\mathrm{L}}:\mathrm{G}\times \mathbb{M}^*\to \mathbb{M^*},
  \]
such that $\mathrm{L}_{\pi_{\mathcal{C}}(g)}\circ \mathrm{p}_{\mathbb{M}}=\mathrm{p}_{\mathbb{M}}\circ \widehat{\mathrm{L}}_g$.
This implies that $\mathrm{G}$ acts by orientation and time-orientation preserving conformal automorphisms.
Thus, $\mathrm{dim}(\mathcal{C}^{\uparrow}_+(\mathbb{M}^*))=(n+3)(n+2)/2$.
By the first part of the proof, we may conclude that $\mathbb{M}^*$ can be identified with
$\E$ {and $\mathrm{G}$ with} $\mathcal{C}^{\uparrow}_+(\E)\cong \OO(2,n+1)$.
In addition, the action of $\OO(2,n+1)$ descends to an action on $\mathbb{M}=\E/\Gamma$.
This implies that, for every $(\mathrm{\mathbf{Z}},\tau')\in \Gamma$ and every
$(\mathrm{\mathbf{X}},\tau)\in \OO(2,n+1)$, we have
\[
 \mathrm{p}_{\mathbb{M}}\circ (\widehat{\mathrm{L}}_{(\mathrm{\mathbf{X}},\tau)}
 \circ \widehat{\mathrm{L}}_{(\mathrm{\mathbf{Z}},\tau')}
 \circ \widehat{\mathrm{L}}_{(\mathrm{\mathbf{X}},\tau)^{-1}})
 = \mathrm{L}_{(\mathrm{\mathbf{X}},\tau)}\circ \mathrm{p}_{\mathbb{M}}
 \circ \widehat{\mathrm{L}}_{(\mathrm{\mathbf{X}},\tau)^{-1}}=
\mathrm{p}_{\mathbb{M}}.
   \]
Then, for every $(\mathrm{\mathbf{Z}},\tau')\in \Gamma$, the image of the map
\[
 f_{(\mathrm{\mathbf{Z}},\tau')}: \OO(2,n+1)\ni (\mathrm{\mathbf{X}},\tau)\longmapsto
(\mathrm{\mathbf{X}},\tau)\star (\mathrm{\mathbf{Z}},\tau')\star (\mathrm{\mathbf{X}},\tau)^{-1}\in \OO(2,n+1)
   \]
is contained in $\Gamma$. Since $\Gamma$ is discrete and $\OO(2,n+1)$ is connected,
$f_{(\mathrm{\mathbf{Z}},\tau')}$ is constant, equal to $(\mathrm{\mathbf{Z}},\tau')$.
Therefore, $\Gamma$ is a subgroup of the center $\widehat{\mathrm{Z}}(2,n+1)$ of $\OO(2,n+1)$.
According to Theorem \ref{ThmA}, we have the following.
\begin{itemize}
\item If $n$ is even, $\Gamma = \left\{(I,2\pi mk)\mid m\in \Z\right\}$, for some positive integer $k$.
Hence, $\mathbb{M}=\E/\Gamma=\mathcal{E}^{1,n}_{\mathrm{I},k}$.

\item If $n$ is odd, then either $\Gamma = \left\{(I,2\pi mk) \mid m\in \Z\right\}$, for some
positive integer $k$, or
$\Gamma = \left\{((-1)^m I,\pi m(1+2k))\mid m\in \Z\right\}$, for some positive integer $k$.
In the first case, $\mathbb{M}=\E/\Gamma=\mathcal{E}^{1,n}_{\mathrm{I},k}$; in the
second case,  $\mathbb{M}=\E/\Gamma=\mathcal{E}^{1,n}_{\mathrm{II},k}$.
\end{itemize}

This concludes the proof of Theorem C.
\end{proof}

\bibliographystyle{amsalpha}

\end{document}